\documentclass[11pt]{article}

\usepackage{amssymb,amsmath,euscript,bbm,xcolor,graphicx,epstopdf}
\usepackage{verbatim} 
\usepackage[titletoc]{appendix}
\usepackage{authblk}
\usepackage{float}
\newtheorem{proposition}{Proposition}[section]
\newtheorem{lemma}[proposition]{Lemma}
\newtheorem{theorem}[proposition]{Theorem}
\newtheorem{definition}[proposition]{Definition}

\newtheorem{condition}{Condition}
\newtheorem{alg}{Algorithm}

\def\S{{\mathbb S}}

\def\E{{\mathbb E}}

\newcommand{\reffig}[1]{Fig.\,(\ref{#1})}
\newcommand{\refeq}[1]{Eqn.\, (\ref{#1})}
\newcommand{\refsec}[1]{Section \ref{#1}}

\newcommand{\healpix}{HEALpix~}

\usepackage{booktabs}
\usepackage{siunitx}

\makeatletter \@addtoreset{equation}{section} \makeatother
\newcommand {\qed}%
{%
    {}\hfill
    {}\hfill
    {$\square $}%
    \vspace {0.3cm}%
    \pagebreak [2]%
    \par
}%
\newenvironment{proof}[1]{%
    \vspace{0.3cm}%
    \pagebreak [2]%
    \par%
    \noindent {\bf  Proof~#1\ }}{\qed}%
\newenvironment{remark}{%
    \vspace{0.3cm} \pagebreak [2]%
    \par%
    \refstepcounter{proposition}
    \noindent%
    {\bf Remark~\theproposition\  }}{}%
\begin{document}

\title{Multiple testing of local maxima\\ for detection of peaks on the
(celestial) sphere}

\author[1]{Dan Cheng}
\author[2]{Valentina Cammarota}
\author[2]{Yabebal Fantaye}
\author[2]{Domenico Marinucci}
\author[1]{Armin Schwartzman}

\affil[1]{Division of Biostatistics, University of California, San Diego}
\affil[2]{Department of Mathematics, University of Rome Tor Vergata}


\maketitle

\begin{abstract}
We present a topological multiple testing scheme for detecting peaks on the
sphere under isotropic Gaussian noise, where tests are performed at local
maxima of the observed field filtered by the spherical needlet transform.
Our setting is different from the standard Euclidean/large same asymptotic framework, yet highly relevant to realistic experimental circumstances for some important areas of application in astronomy. More precisely, we  focus on cases where a single realization of a smooth isotropic Gaussian random field on the sphere is observed, and a number of well-localized signals are superimposed on such background field.
The proposed algorithms, combined with the Benjamini-Hochberg procedure for
thresholding p-values, provide asymptotic strong control of the False
Discovery Rate (FDR) and power consistency as the signal strength and the
frequency of the needlet transform get large. This novel multiple testing method
is illustrated in a simulation of point-source detection in Cosmic Microwave Background radiation (CMB) data.

\begin{itemize}
\item \textbf{Keywords and Phrases: }{\small Gaussian random fields; Sphere;
CMB; Height distribution; Overshoot distribution; Needlet transform;
P-value; Threshold; False discovery rate; Power.}

\item \textbf{AMS Classification:} 60G15; 60G60; 62H15; 62M15; 62M40
\end{itemize}
\end{abstract}

\noindent

\section{Introduction}

A classical problem of modern high-dimensional statistics is multiple
testing in the presence of background noise. Applications are common in the
areas of neuroimaging, genomic arrays and astronomy. These issues become
particularly challenging when the background noise is allowed to exhibit
more realistic properties than the simple $i.i.d.$ framework, in particular
when noise is modeled as a stochastic process or a random field. In this
setting, important progresses have been recently obtained combining ideas
from two different streams of research, namely techniques from the multiple
testing literature, such as False Discovery Rate (FDR) algorithms, and
techniques to investigate excursion probabilities and local maxima for
random fields; we refer for instance to \cite{RFG,Schwartzman:2011,chengschwartzman1,chengschwartzman2}
for further background and discussion. These works have covered
applications in a univariate and multivariate Euclidean setting; analytic
properties have been derived under a large sample asymptotic framework,
i.e., under the assumption that the domain of observations is growing
steadily, together with the signals to be detected.

In this paper we introduce a related multiple testing procedure in a setting that is different from the standard Euclidean/large same asymptotic framework, yet highly relevant to realistic experimental circumstances for some important areas of application in astronomy. More precisely, we shall focus on cases where a single realization of a smooth isotropic Gaussian random field on the sphere is observed, and a number of well-localized signals are
superimposed on such background field.
This is exactly the setting for the so-called
point-source detection issue in Cosmic Microwave Background radiation (CMB)
data experiments (see i.e., \cite{planck,PlanckIS,PlanckNG,PlanckCS}). As discussed now in any modern textbook in Cosmology (see
for instance \cite{dode2004,Durrer}), CMB data can be viewed as a
single realization of an isotropic Gaussian random field, which represents a
\textquotedblleft snapshot" of the Universe at the \emph{last scattering
surface, }i.e. the time (approximately $4\times 10^{5}$ years after the Big
Bang, or $1.38\times 10^{10}$ years ago) when photons decoupled from
electrons and protons and started to travel nearly without interactions in
space. As such, CMB has been repeatedly defined as a goldmine of information
on Cosmology - two very successful satellite experiments (WMAP from NASA,
see http://map.gsfc.nasa.gov/ and Planck from ESA, see
http://www.esa.int/Our\_Activities/Space\_Science/Planck) have now produced
full-sky maps of CMB radiations, and these data have been used in several
thousand papers over the last few years to address a number of fundamental
questions on the dynamics of the Big Bang, the matter-energy content of the
Universe, the mechanisms of structure formation, and several others.

From the experimental point of view, it is very important to recall that,
superimposed \ to CMB radiation, a number of foreground \textquotedblleft
contaminants" are present; as a first approximation, we can view them as
point-like objects (galaxies or clusters of galaxies, typically). A major
statistical challenge in the analysis of CMB data is the proper
identification of such sources; on the one hand this is important for the
proper construction of filtered CMB maps, on the other hand these sources
are of great interest on their own as proper astrophysical objects (in
some cases they can be matched with existing catalogues, while in other cases
they lead to new discoveries). A number of algorithms have been proposed for
these tasks, see for instance \cite{argueso,axelsson,marcos2006,scodeller,scodeller2}.
These solutions have all been shown to perform well in practice; however, they have all avoided to face
the specific challenges of multiple testing, and in particular none of them
has been shown to control in any proper statistical way any aggregate
statistics such as the classical Family-Wise Error Rate (FWER), False Discovery Proportion (FDP) or False Discovery Rate (FDR).

Our purpose in this paper is to develop in such a spherical framework
a rigorous statistical procedure to control error rates in a multiple
testing framework. Our starting idea is to extend to these
circumstances the Smoothing and TEsting of Maxima (STEM) algorithm
advocated in \cite{Schwartzman:2011,chengschwartzman2}, and
investigate rigorously its statistical properties. While our
construction follows in several ways what was earlier done by these
authors, we wish to stress that the new spherical framework poses some
major technical and foundational new challenges.

The first of these new challenges is the proper definition of filters and
point-like signals in a spherical framework. Here, natural solutions can be
found by exploiting recent developments in the analysis of spherical random
fields an spherical wavelets. In particular, we can define bell-shaped
signals by adopting a natural definition of a Gaussian distribution on the
sphere, motivated in terms of diffusion processes; likewise, filtering can
be implemented by wavelet techniques - we find particularly convenient the
Mexican needlet construction introduced by \cite{gm1,gm2}, see also \cite{scodellermexican} for some
earlier applications to CMB data.

A second, more delicate, issue is the rigorous investigation
of asymptotic statistical properties. A crucial staple of the
STEM algorithm is the possibility to control the FDR, assuming
convergence of the empirical distribution of the maxima to its theoretical
counterpart. In standard settings this can be done
by resorting to ergodicity properties in a ``large
sample asymptotics" framework, i.e. assuming that the domain of the
observations grows larger and larger. This form of ergodic properties
cannot be exploited here because our spherical domain is compact. We shall hence require a convergence result on the empirical distribution of maxima in a high-frequency/fixed domain setting: this result extends to the case of the needlet transform some related computations which were recently performed in \cite{cmw2014} for the case of random spherical harmonics. In this sense, our setting is related to the increasingly popular fixed-domain asymptotics approach for the analysis of random fields, see i.e., \cite{loh,loh2}.

The plan of this paper is as follows. Our basic setting and model is
introduced in Section \ref{S:themodel}, with Section \ref{S:filtering}
devoted to a careful discussion on the nature and effects of filtering.
Section \ref{S:multipletesting} provides a description of the multiple
testing scheme and discusses the error and power definitions, including the
derivation of asymptotic  $p$-values and our adoption of Benjamini and
Hochberg's (1995) pioneering approach \cite{BH}. The proofs of FDR
control and tests consistency are collected in Section \ref{S:proofFDR},
while in Section \ref{numerical} we provide some numerical results on the
empirical performance of the proposed procedures in simulated CMB fields.
Finally,  two (long) Appendixes provide a large part of the proofs, in particular,
the details of the high frequency ergodicity of the empirical distribution function
of local maxima.

For clarity of exposition
and concreteness of motivations, throughout the paper we mainly justify our
framework resorting to applications in a Cosmological framework; it is to be
stressed, however, that our approach may be applied to many other
experimental contexts where data are collected on a sphere, such as
Geophysics, Atmospheric Sciences, Solar Physics, and even (with some
approximations) Neuroimaging, to mention only a few.

\subsection{Acknowledgements}  This research was supported by the ERC Grant n.277742 \emph{Pascal} (PI Domenico Marinucci); we are grateful to Igor Wigman for many insights and suggestions on the computation of variances of critical points. Corresponding author: Valentina Cammarota, Department of Mathematics, University of Rome Tor Vergata, via della Ricerca Scientifica, 1, I-00133, Italy, cammarot@mat.uniroma2.it.

\section{The model \label{S:themodel}}

The purpose of this Section is to introduce our model in detail. As
motivated above, our purpose here is to represent a situation where a large
number of \textquotedblleft point sources" is superimposed on some isotropic
background \textquotedblleft noise". Of course, the notions of
\textquotedblleft noise" and \textquotedblleft signal", here as in any other
motivating field, is very much conventional. For instance, in the
CMB-related applications that we have in mind the background Gaussian field
is eventually the primary object of physical interest for many (most)
researchers, while the super-imposed point sources are contaminants to be
removed; in other astrophysical areas, on the contrary, the identification
of the sources may be by itself a major scientific goal (a very recent
catalogue of detected point sources/astrophysical objects is given for
instance by \cite{PlanckCS}).

\subsection{Signal model}
\label{sec:signal}

To introduce our model for the signal, we first need to justify the notion
of a \textquotedblleft Gaussian-shaped" density on the sphere.
We do this in terms of the diffusion equation on the sphere.
More formally, let ${\mathbb{S}}^{2}$ denote the unit sphere in $\mathbb{R}^{3}$.
The diffusion equation on the sphere is then given by%
\begin{align*}
\frac{\partial }{\partial t}h(x;t,x_{0}) &= -\Delta _{S^{2}\text{ }}h(x;t,x_{0}) \\
h(x;0,x_{0}) &= \delta _{x_{0}}(x),
\end{align*}
where $\Delta _{S^{2}\text{ }}$ is the Laplacian operator in ${\mathbb{S}}^{2}$ and $\delta _{x_{0}}(x)$ identifies formally a Dirac's delta function
centered at $x_{0}\in S^{2}$. It is standard to write the solution in terms
of diffusion operators as
\begin{equation*}
h(x;t,x_{0})=\exp (-t\Delta _{S^{2}\text{ }})h(x;0,x_{0})=\sum_{\ell }\frac{%
2\ell +1}{4\pi }\exp (-t\lambda _{\ell })P_{\ell }(\left\langle
x_{0},x\right\rangle )\text{, }
\end{equation*}
where $\lambda _{\ell }:=\ell (\ell +1)$ denotes the set of eigenvalues of
the spherical Laplacian, $\ell =1,2,...,$ while $\left\{ P_{\ell
}(\cdot)\right\} $ represents the family of Legendre polynomials%
\begin{equation*}
P_{\ell }(u) = \frac{(-1)^{\ell }}{2^{\ell }}\frac{d^{\ell }}{du^{\ell }}%
(1-u^{2})^{\ell }, \qquad \ell =1,2,3,\ldots
\end{equation*}%
i.e., $P_{1}(u) = u$, $P_{2}(u)=(3u^{2}-1)/2$, $P_{3}(u)=(5u^{3}-3u)/2$, etc.,
and $\left\langle \cdot,\cdot\right\rangle $ denotes inner product on ${\mathbb{S}}%
^{2}.$ By a straightforward analogy with the Euclidean case, it is
natural/customary to view $u_{x_{0}}(\cdot,t)$ as the density on ${\mathbb{S}}%
^{2}$ of a \textquotedblleft spherical Gaussian" centred on $x_{0}$ and
having variance $t$.

Our \textquotedblleft point source" signal will be built from a set of
such bell-shaped distributions, localized on some family of points $\xi _{k}$%
, $k=1,\ldots ,N$, belonging to ${\mathbb{S}}^{2}$; we shall allow later
their number to grow ($N=N(j)\rightarrow \infty $ as $j\rightarrow \infty $,
where $j$ is the frequency), and their shape to become more and more
localized ($t=t_{N(j)}\rightarrow 0$ as $N(j)\rightarrow \infty )$. These
conditions can be understood by an analogy with the (now standard)
high-dimensional asympotics framework where the number of parameters is
allowed to grow to infinity in the presence of growing number of observations;
likewise, we consider a growing number of sharper and sharper sources as the
resolution of our experiments grow better and better, or equivalently as the
scales that we are able to probe become smaller and smaller. It should be
noted that, in the absence of these conditions, all our procedures to follow
have properties that can be trivially established: in particular, the power
of our detection procedures is very easily seen to converge to unity. More
explicitly, we believe that our setting is meaningful and relevant as a
guidance for applied scientists; indeed, in many circumstances (such as the
CMB data analysis framework that we mentioned several times) the number of
tests to implemented (i.e., the number of possible galactic sources) is in
the order of several thousands, so it seems more useful to consider
this quantity as diverging to infinity together with the number of
observations.

\subsection{Signal plus noise model}
\label{sec:signal+noise}

We can hence introduce the following sequence of signal-plus-noise models,
for $N=1,2,\ldots$
\begin{equation}
y_{N}(x)=\mu _{N}(x)+z(x),\qquad x\in {\mathbb{S}}^{2},
\label{eq:signal+noise}
\end{equation}%
where the $\mu _{N}(x)$ denotes a sequence of deterministic functions on the
sphere defined by
\begin{equation}
\mu _{N}(x)=\sum_{k=1}^{N}a_{k}h(x;t_{N},\xi _{k}),\qquad a_{k}>0,
\label{eq:mu}
\end{equation}%
and $h(x;t_{N},\xi _{k})$ is the family of \textquotedblleft spherical
Gaussian distributions" on ${\mathbb{S}}^{2}$ (centred on $\xi _{k}$ and
with variance $t_{N})$ which we introduced above by means of the heat kernel
on ${\mathbb{S}}^{2}$, i.e.,
\begin{equation*}
h(x;t_{N},\xi _{k})=\sum_{\ell =0}^{\infty }\exp (-\ell (\ell +1)t_{N})\frac{%
2\ell +1}{4\pi }P_{\ell }(\left\langle \xi _{k},x\right\rangle )\text{ }.
\end{equation*}%
As mentioned earlier, we will set $t_{N}\rightarrow 0$ as $N\rightarrow
\infty $, so that each kernel $h(x;t_{N},\xi _{k})$ will become in the limit
more and more concentrated around its center $\left\{ \xi _{k}\right\} $.

Let us now focus on the \textquotedblleft noise" component $z(x);$ here, we
need to recall briefly a few standard facts on the harmonic representations
of isotropic, finite variance spherical random fields. In particular, let us
assume that $\left\{ z(x),\text{ }x\in {\mathbb{S}}^{2}\right\} $ is
Gaussian, zero-mean and isotropic, meaning that the probability laws of $%
z(\cdot )$ and $z^{g}(\cdot ):=z(g\cdot )$ are the same for any rotation $%
g\in SO(3)$. For such fields, it is well-known that the following
representation holds in the mean square sense (see for instance \cite%
{leonenko2}, \cite{marpecbook}):
\begin{equation}
z(x)=\sum_{\ell =1}^{\infty }z_{\ell }(x), \qquad z_{\ell
}(x)=\sum_{m=-\ell }^{\ell }a_{\ell m}Y_{\ell m}(x),  \label{specrap}
\end{equation}
where $\left\{ Y_{\mathbb{\ell }m}(.)\right\} $ denotes the family of
spherical harmonics (see for instance \cite{marpecbook}, Chapters 3 and 5), and $%
\left\{ a_{\mathbb{\ell }m}\right\} $ denotes the array of random spherical harmonic
coefficients, which satisfy
\begin{align}
a_{\ell m} &= \int_{{\mathbb{S}}^{2}}z(x)\overline{Y}_{\ell m}(x)dx \\
\mathbb{E}a_{\mathbb{\ell }m}\overline{a}_{%
\mathbb{\ell }^{\prime }m^{\prime }} &= C_{\mathbb{\ell }}\delta _{\mathbb{\ell
}}^{\mathbb{\ell }^{\prime }}\delta _{m}^{m^{\prime }};
\label{eq:alm}
\end{align}%
here, $\delta_{a}^{b}$ is the Kronecker delta function, and the sequence $%
\left\{ C_{\mathbb{\ell }}\right\} $ represents the so-called angular power
spectrum of the field. As pointed out in \cite{mp2012}, under isotropy and
finite-variance the sequence $C_{\mathbb{\ell }}$ necessarily satisfies $%
\sum_{\mathbb{\ell }}C_{\ell }\frac{(2\mathbb{\ell }+1)}{4\pi }=\mathbb{E}%
\left[ z^{2}(x)\right] <\infty $ and the random field $z(x)$ is mean square
continuous. Its covariance function is given by
\begin{equation*}
\Gamma (x_{1},x_{2})={\mathbb{E}}\left[ z(x_{1})z(x_{2})\right] =\sum_{\ell
=0}^{\infty }\frac{2\ell +1}{4\pi }C_{\ell }P_{\ell }(\left\langle
x_{1},x_{2}\right\rangle ), \qquad x_{1},x_{2}\in {\mathbb{S}}^{2}.
\end{equation*}%
The Fourier components $\left\{ z_{\mathbb{\ell }}(x)\right\} $, can be
viewed as random eigenfunctions of the spherical Laplacian:%
\begin{equation*}
\Delta _{S^{2}}z_{\mathbb{\ell }}=-\mathbb{\ell }(\mathbb{\ell }+1)z_{%
\mathbb{\ell }}, \qquad \mathbb{\ell }=1,2,\ldots;
\end{equation*}%
the asymptotic behaviour of $z_{\ell }(x)$ and their nonlinear transforms
has been studied for instance by \cite{Wig1}, \cite{Wig2} and \cite{MaWi3}.

\section{Filtering and smoothing \label{S:filtering}}

An important step in the implementation of the STEM algorithm is kernel
smoothing of the observed data. Given the very delicate nature of the asymptotic
results in our setting, the definition of the kernel function requires here
special care. We shall propose here to adopt a kernel which is based upon
the so-called Mexican needlet construction introduced by \cite{gm1,gm3}, see also \cite{spalan,mayeli,scodellermexican} for the
investigation of stochastic properties and statistical applications of these techniques.

Mexican needlets can be viewed as a natural development of the standard
needlet frame which was introduced by \cite{npw1,npw2}. Loosely
speaking, Mexican needlets differ from the standard needlet construction inasmuch
as they allow for providing a kernel which is unboundedly supported in the harmonic
domain; they can hence be shown to enjoy better localization properties in
the real domain, i.e., faster (Gaussian rather than nearly exponential)
decay of their tails. For our purposes, these better localization properties
in the real domain turn out to be very important, as they allow a tight
control of leakage in the signals.

The Mexican needlet transform of order $p\in \mathbb{N}$ can be defined by
\begin{equation}
\Psi _{j}(\langle x_{1},x_{2}\rangle ):=\sum_{\ell =0}^{\infty }b\left(
\frac{\ell }{B^{j}};p\right) \frac{2\ell +1}{4\pi }P_{\ell }(\langle
x_{1},x_{2}\rangle )\text{ ;}  \label{eq:kernel}
\end{equation}%
here, the function $b(\cdot;p)$ is defined by $b(u;p)=u^{2p}e^{-u^{2}},$ with $%
u\in \mathbb{R}_{+},$ and it is easily seen to belong to the Schwartz class
(i.e., all its derivatives decay faster than any polynomial). The
user-chosen integer parameter $p$ can be taken for simplicity to be equal to
unity for all the developments that follow; more generally, it has been
shown that higher values of $p$ entail better properties in the harmonic
domain, but worse real-space localization: in particular, a higher number of
sidelobes (see for instance \cite{scodellermexican} for discussion and
numerical evidence on these issues).

Let us now recall the standard addition theorem for spherical harmonics (see
\cite{marpecbook}, eq. 3.42)
\begin{equation*}
\sum_{m=-\ell }^{\ell }Y_{\ell m}(x_{1})\overline{Y}_{\ell m}(x_{2})=\frac{%
2\ell +1}{4\pi }P_{\ell }(\langle x_{1},x_{2}\rangle );
\end{equation*}%
it is then easy to see that the ``filtered noise" is given by
\begin{equation}\label{eq:betaj}
\begin{split}
\beta _{j}(x) &:=\langle \Psi _{j}(\langle x,y\rangle ),z(y)\rangle _{L^{2}({\mathbb{S}}
^{2})}\\
& =\int_{{\mathbb{S}}^{2}}\sum_{\ell =0}^{\infty }b\left( \frac{\ell
}{B^{j}};p\right) \frac{2\ell +1}{4\pi }P_{\ell }(\langle x,y\rangle )z(y)dy
\\
& =\int_{{\mathbb{S}}^{2}}\sum_{\ell =0}^{\infty }b\left( \frac{\ell }{B^{j}}
;p\right) \sum_{m=-\ell }^{\ell }Y_{\ell m}(x)\overline{Y}_{\ell m}(y)z(y)dy
\\
& =\sum_{\ell =0}^{\infty }b\left( \frac{\ell }{B^{j}};p\right)
\sum_{m=-\ell }^{\ell }a_{\ell m}Y_{\ell m}(x)=\sum_{\ell =0}^{\infty
}b\left( \frac{\ell }{B^{j}};p\right) z_{\ell }(x),
\end{split}
\end{equation}
where the last line is due to (\ref{specrap}). On the other hand, for the
``filtered signal" we obtain
\begin{equation}\label{eq:muj}
\begin{split}
\mu _{N,j}(x)&:=\sum_{k=1}^{N} a_{k}\langle \Psi _{j}(\langle x,y\rangle ),h(y;t_{N},\xi _{k})\rangle _{L^{2}({%
\mathbb{S}}^{2})}\\
& =\sum_{k=1}^{N} a_{k}\int_{{\mathbb{S}}^{2}}\sum_{\ell =0}^{\infty }b\left(
\frac{\ell }{B^{j}};p\right) \frac{2\ell +1}{4\pi }P_{\ell }(\langle
x,y\rangle )h(y;t_{N},\xi _{k})dy \\
& =\sum_{k=1}^{N} \sum_{\ell =0}^{\infty }a_{k} b\left( \frac{\ell }{B^{j}};p\right) \exp (-\ell
(\ell +1)t_{N})\frac{2\ell +1}{4\pi }P_{\ell }(\left\langle \xi
_{k},x\right\rangle ).
\end{split}
\end{equation}
In words, both the filtered noise and signals are averaged versions, in the
harmonic domain, of (random and deterministic, respectively) Fourier
components. Summing up, our kernel transform produces the sequence of smoothed fields
\begin{equation}
y_{N,j}(x):=\mu _{N,j}(x)+\beta _{j}(x),  \label{eq:conv}
\end{equation}%

Our asymptotic theory will be developed in the so-called ``high-frequency"
framework; more precisely, we shall introduce the following assumptions.

\begin{condition}
\label{C:tN&Bj} We have that, as $j\rightarrow \infty $%
\begin{equation*}
\frac{t_{N}}{B^{2j}}=\frac{t_{N}(j)}{B^{2j}}\rightarrow 0\text{ . }
\end{equation*}
\end{condition}
In words, we are assuming that both the filter and the signal become more
and more localized, the former more rapidly to make identification
meaningful. Before we discuss these assumptions, however, we need to explore
in greater detail the properties of these two components; this task is
implemented in the next two subsections.

\subsection{The filtered signal}
\label{sec:localization}

For the analysis of the signal component $\left\{ \mu _{N,j}(\cdot)\right\} ,$
it is convenient to introduce the simple approximation
\begin{equation}
\begin{split}
\mu _{N,j}(x)& :=\sum_{k=1}^{N}\sum_{\ell =0}^{\infty }a_{k}\left( \frac{%
\ell }{B^{j}}\right) ^{2p}\exp \left( -\ell (\ell +1)t_{N}-B^{-2j}\ell
^{2}\right) \frac{2\ell +1}{4\pi }P_{\ell }(\left\langle \xi
_{k},x\right\rangle ) \\
& =\sum_{k=1}^{N}\sum_{\ell =0}^{\infty }a_{k}\left( \frac{\ell }{B^{j}}%
\right) ^{2p}\exp \left( -B^{-2j}\ell ^{2}\right) \frac{2\ell +1}{4\pi }%
P_{\ell }(\left\langle \xi _{k},x\right\rangle )+o_{j}(1) \\
& =\sum_{k=1}^{N}a_{k}\Psi _{j}(\langle x,\xi _{k}\rangle )+o_{j}(1),
\end{split}
\label{eq:mujM}
\end{equation}%
where the second line can be easily justified resorting to Condition \ref%
{C:tN&Bj} above.
It is also known that the smoothing kernel $\Psi _{j}(\langle \cdot ,\xi
_{k}\rangle )$, for any $\xi _{k}\in {\mathbb{S}}^{2},$ has Gaussian tails
(up to a polynomial factor), and hence decays faster than exponentially;
more precisely one has that there exists a constant $C_{p}$ such that (\cite{gm1}, \cite{gm2}, \cite{gm3})
\begin{equation}
|\Psi _{j}(\langle x,\xi _{k}\rangle )|\leq C_{p}B^{2j}e^{-\frac{%
B^{2j}d^{2}(x,\xi _{k})}{4}}\left( 1+\left\vert H_{2p}\left( B^{j}d(x,\xi
_{k})\right) \right\vert \right) ,  \label{eq:lp-Mexican}
\end{equation}%
where $d(x,y)=\arccos (\langle x,y\rangle )$ is the standard geodesic
distance on the sphere and  $H_{q}(\cdot )$ denotes the Hermite polynomial of degree $q$, which is
defined by
\begin{equation*}
H_{q}(x)=(-1)^{q}e^{x^{2}/2}\frac{d^{q}}{dx^{q}}\left( e^{-x^{2}/2}\right) ,
\end{equation*}%
the first few being $H_{1}(x)=x$, $H_{2}(x)=x^{2}-1$, $H_{3}(x)=x^{3}-3x$,...
It is also possible to provide a useful analytic approximation for the
functional form of the needlet filter at the highest frequencies $j;$ indeed
Geller and Mayeli (2009) \cite{gm1} proved the following.

\begin{lemma}
\label{GellerMayeli} Let $p=1$ and let $\xi _{k}\in {\mathbb{S}}^{2}$ be
fixed. Then as $j\rightarrow \infty $,
\begin{equation*}
\Psi _{j}(\langle x,\xi _{k}\rangle )=g(d(x,\xi _{k}))(1+O(B^{-2j})),
\end{equation*}%
where
\begin{equation*}
g(\theta )=\frac{1}{4\pi }B^{2j}e^{-\frac{B^{2j}\theta ^{2}}{4}}\left( 1-%
\frac{B^{2j}\theta ^{2}}{4}\right) ,\quad \theta \in \lbrack 0,\pi ]
\end{equation*}
and $d$ is the standard geodesic distance on the sphere.
\end{lemma}

As a consequence, we have the following analytic expression for our signal when $p=1$, as $j\rightarrow \infty $:
\begin{equation}
\mu _{N,j}(x)=\sum_{k=1}^{N}a_{k}\frac{1}{4\pi }B^{2j}e^{-\frac{%
B^{2j}d^{2}(x,\xi _{k})}{4}}\left( 1-\frac{B^{2j}d^{2}(x,\xi _{k})}{4}%
\right) (1+O(B^{-2j})).  \label{eq:lp-signal}
\end{equation}%
It is readily verified that the function $g(\cdot)$ has the global maximum $g(0)=%
\frac{1}{4\pi }B^{2j}$ and a local minimum $g(2\sqrt{2}B^{-j})=-\frac{1}{%
4\pi }e^{-2}B^{2j}$.

\subsection{The filtered noise}

Our next step is to focus on the sequence of filtered noise fields; as
derived above (equation \ref{eq:betaj}), they can be expressed as averaged
forms of random spherical eigenfunctions, e.g.
\begin{equation*}
\beta _j(x)=\sum_{\ell =1}^{\infty }b\Big(\frac{\ell }{B^{j}};p\Big)z_{%
\mathbb{\ell }}(x), \quad j=1,2,3,\ldots, \qquad b(u;p)=u^{2p}e^{-u^{2}}, \quad
u\in \mathbb{R}.
\end{equation*}%
It is convenient to normalize these fields to have unit variance, as
follows:
\begin{equation}
\tilde{\beta}_j(x)=\frac{\beta _j(x)}{\sqrt{\mathbb{E}[\beta
_{j,p}^{2}(x)]}}, \qquad j=1,2,3,\dots  \label{eq:standardize}
\end{equation}%
Let us define also
\begin{equation}
\tilde{y}_{N,j}=\frac{y_{N,j}}{\sqrt{{\mathbb{E}}[\beta
_{j,p}^{2}(x)]}}=\tilde{%
\beta}_{j}+\frac{\mu _{N,j}}{\sqrt{{\mathbb{E}}[\beta
_{j,p}^{2}(x)]}}.  \label{eq:signal+noise:norm}
\end{equation}%
A rigorous investigation of the asymptotic properties of these smoothed
fields requires some mild regularity assumptions on the power spectrum $%
C_{\ell }$, which are customary in this branch of literature. More
precisely (see for instance \cite{marpecbook}, page 257 or \cite{bkmpAoS,mpbb08,spalan,mayeli}),

\begin{condition}
\label{con}There exists $M\in \mathbb{N},\gamma >2$ and a function $G(\cdot)\in
C^{\infty }$ such that%
\begin{equation}\label{eq:Cl}
C_{\ell }=\ell ^{-\gamma }G(\ell )
\end{equation}%
where $0<G(\ell )$ for all $\ell ,$ and for some $c_{1},\dots ,c_{M}>0$ and $%
r=1,\dots ,M$, we have
\begin{equation*}
\sup_{u}\left|\frac{d^{r}}{du^{r}}G(u)\right|\leq c_{r}u^{-r}.
\end{equation*}
\end{condition}

Condition \ref{con} entails a weak smoothness requirement on the behaviour
of the angular power spectrum, which is satisfied by cosmologically relevant
models; for instance, this condition is fulfilled by models of the form
\eqref{eq:Cl}, where $G(\ell )={P(\ell )}/{Q(\ell )}$ and $P(\ell ),Q(\ell )>0$ are two
positive polynomials of the same order. In what follows we denote by $G_{0}$
the limit $G_{0}:=\lim_{\ell \rightarrow \infty }G(\ell )$.

Under condition \ref{con}, it is possible to establish an upper bound on the
correlation function of $\{\tilde{\beta}_j(\cdot )\}$, as follows
(for a proof see \cite{spalan}, \cite{mayeli}). 

\begin{proposition}
\label{con1} Assume Conditions \ref{con} holds with $\gamma <4p+2$ and $%
M\geq 4p+2-\gamma $; then there exists a constant $K_{M}>0$, not depending
on $j$, $x$, and $y$, such that the following inequality holds
\begin{equation}
|\mathrm{Cor}(\tilde{\beta}_j(x),\tilde{\beta}_j(y))|\leq \frac{K_{M}%
}{(1+j^{-1}B^{j}d(x,y))^{4p+2-\gamma }},  \label{corrineq}
\end{equation}%
where $d(x,y)=\arccos (\langle x,y\rangle )$ is the standard geodesic
distance on the sphere.
\end{proposition}

The inequality \eqref{corrineq} is qualitatively similar to others which
were earlier established in the case of standard needlets; see for instance
\cite{bkmpAoS}. A quick comparison with the results in \cite{bkmpAoS} shows an
important difference, namely that the rate of decay for the bound on the
right-hand side depends on the shape of the kernel (in particular, on
the parameter $p$) and on the rate of decay of the angular power spectrum
(i.e., on the parameter $\gamma $); none of these values affect the rate
of convergence in the standard needlet case. As a consequence, in the case
of Mexican needlets, asymptotic uncorrelation only holds under the assumption
that $\gamma <4p+2,$ so that higher values of $p$ are needed to ensure
uncorrelation for larger values of $\gamma .$ We believe this issue can be
easily addressed by a plug-in procedure; for instance, for the CMB
applications we mentioned earlier there are strong theoretical motivations
and experimental constraints that allow to set $2<\gamma <3,$ so that taking
$p=1$ is already enough to ensure the correlation function decays to zero:
ample numerical evidence on the uncorrelation properties of Mexican needlets
is collected in \cite{scodellermexican}. The term $j^{-1}$ appearing in the
denominator of \eqref{corrineq} is a consequence of some standard technical difficulties when dealing with boundary cases such as $M=4p+2-\gamma$.

Of course, from Proposition \ref{con1} it is immediate to obtain a bound on
the covariance (rather than correlation) function, indeed we have
\begin{align*}
\Gamma _{j,p}(x,y) :=& \mathbb{E}[\beta _j(x)\beta _j(y)]=\sum_{\ell
=1}^{\infty }b^{2}\Big(\frac{\ell }{B^{j}};p\Big)C_{\ell }\frac{2\ell +1}{%
4\pi }P_{\ell }(\langle x,y\rangle ) \\
\leq &\frac{K_{M}}{(1+j^{-1}B^{j}d(x,y))^{4p+2-\gamma }}\sum_{\ell
=1}^{\infty }b^{2}\Big(\frac{\ell }{B^{j}};p\Big)C_{\ell }\frac{2\ell +1}{%
4\pi }.
\end{align*}%

For the implementation of our multiple testing procedures, we shall need to
write down an analytic formula for the asympotic distribution of maxima of
the noise components; to this aim, we need the exact limiting behaviour of
higher-order derivatives of the covariance function, evaluated at the
origin. Let us first introduce the functions%
\begin{equation*}
c_{p,2n}(\gamma ):=2^{\gamma /2-2-n-2p}\Gamma (1-\gamma /2+n+2p), \qquad
\Gamma (t):=\int_{0}^{\infty }x^{t-1}\exp (-x)dx\text{ .}
\end{equation*}%
For an isotropic Gaussian field $\{X(x),$ $x\in {\mathbb{S}}^{2}\}$ with
covariance function
\begin{equation*}
C(x,y)=\sum_{\ell =1}^{\infty }\frac{2\ell +1}{4\pi }C_{\ell }P_{\ell }({%
\langle }x,y\rangle ),
\end{equation*}
we define
\begin{equation*}
C^{\prime }(X):=\sum_{\ell =1}^{\infty }\frac{2\ell +1}{4\pi }C_{\ell
}P_{\ell }^{\prime }(1),\qquad C^{\prime \prime }(X):=\sum_{\ell =1}^{\infty }%
\frac{2\ell +1}{4\pi }C_{\ell }P_{\ell }^{\prime \prime }(1),
\end{equation*}%
where
\begin{equation*}
P_{\ell }^{\prime }(1)=\frac{\ell (\ell +1)}{2} \quad\text{ and }\quad P_{\ell }^{\prime
\prime }(1)=\frac{\ell (\ell -1)(\ell +1)(\ell +2)}{8}
\end{equation*}%
represent the derivatives of the Legendre polynomials evaluated at 1. Let us write also
\begin{equation}
\kappa _{1,j}=\frac{C^{\prime }(\tilde{\beta}_{j})}{C^{\prime \prime }(%
\tilde{\beta}_{j})},\qquad \kappa _{2,j}=\frac{[C^{\prime }(\tilde{\beta}%
_{j})]^{2}}{C^{\prime \prime }(\tilde{\beta}_{j})}.
\label{eq:kappa}
\end{equation}

\begin{proposition}
\label{prop:kappa} \bigskip As $j\rightarrow \infty ,$ we have%
\begin{equation*}
C^{\prime }(\tilde{\beta}_{j})\sim \frac{c_{p,2}(\gamma )}{2c_{p,0}(\gamma )}%
B^{2j},\quad C^{\prime \prime }(\tilde{\beta}_{j})\sim \frac{c_{p,4}(\gamma )%
}{8c_{p,0}(\gamma )}B^{4j},
\end{equation*}%
and therefore
\begin{equation*}
\kappa _{1,j}\sim \frac{4c_{p,2}(\gamma )}{c_{p,4}(\gamma )}B^{-2j},\hspace{%
1cm}\kappa _{2,j}\sim \frac{2c_{p,2}^{2}(\gamma )}{c_{p,0}(\gamma
)c_{p,4}(\gamma )};
\end{equation*}%
where $a_{j}\sim b_{j}$ denotes lim$_{j\rightarrow \infty }a_{j}/b_{j}=1.$
\end{proposition}

We note that the expressions for $\kappa _{1,j}$ and $\kappa _{2,j}$ will be
used in the applied sections below for the numerical evaluation of $p$%
-values.

\begin{proof}\
Following the notation in \cite{chengschwartzman1}, we have

\begin{equation*}
C^{\prime }(\tilde{\beta}_{j})=\frac{1}{\mathrm{Var}(\beta _{j})}\sum_{\ell
=0}^{\infty }b^{2}\left( \frac{\ell }{B^{j}};p\right) C_{\ell }\frac{2\ell +1%
}{4\pi }P_{\ell }^{\prime }(1)
=\frac{1}{\mathrm{Var}(\beta _{j})}\sum_{\ell =0}^{\infty }b^{2}\left( \frac{%
\ell }{B^{j}};p\right) C_{\ell }\frac{2\ell +1}{4\pi }\frac{\ell (\ell -1)}{2%
}
\end{equation*}
and
\begin{equation*}
\begin{split}
C^{\prime \prime }(\tilde{\beta}_{j})& =\frac{1}{\mathrm{Var}(\beta _{j})}%
\sum_{\ell =0}^{\infty }b^{2}\left( \frac{\ell }{B^{j}};p\right) C_{\ell }%
\frac{2\ell +1}{4\pi }P_{\ell }^{\prime \prime }(1) \\
& =\frac{1}{\mathrm{Var}(\beta _{j})}\sum_{\ell =0}^{\infty }b^{2}\left(
\frac{\ell }{B^{j}};p\right) C_{\ell }\frac{2\ell +1}{4\pi }\frac{\ell (\ell
-1)(\ell +1)(\ell +2)}{8}.
\end{split}%
\end{equation*}%
Indeed
\begin{equation*}
P_{\ell }^{\prime \prime }(1)=\frac{\ell (\ell -1)(\ell +1)(\ell +2)}{8}=%
\frac{\ell (\ell +1)[\ell (\ell +1)-2]}{8}=\frac{\ell ^{2}(\ell +1)^{2}}{8}-%
\frac{\ell (\ell +1)}{4}.
\end{equation*}%
It is also convenient to introduce the notation
\begin{align} \label{B_not}
\mathcal{B}_{2n}=\mathcal{B}_{2n,p,j}& =\frac{1}{\sum_{\ell =1}^{\infty
}b_{p}^{2}(\frac{\ell }{B^{j}})C_{\ell }\frac{2\ell +1}{4\pi }}\sum_{\ell
=1}^{\infty }b^{2}\left( \frac{\ell }{B^{j}};p\right) C_{\ell }\frac{2\ell +1%
}{4\pi }\ell ^{n}(\ell +1)^{n}
 \nonumber\\
& =\frac{1}{\sum_{\ell =1}^{\infty }b_{p}^{2}(\frac{\ell }{B^{j}})\ell
^{-\gamma }G(\ell )\frac{2\ell +1}{4\pi }}\sum_{\ell =1}^{\infty
}b^{2}\left( \frac{\ell }{B^{j}};p\right) \ell ^{-\gamma }G(\ell )\frac{%
2\ell +1}{4\pi }\ell ^{n}(\ell +1)^{n},
\end{align}%
for which we have
\begin{align}
\lim_{j\rightarrow \infty }& B^{j(\gamma -2)}\sum_{\ell =1}^{\infty
}b^{2}\left( \frac{\ell }{B^{j}};p\right) \ell ^{-\gamma }G(\ell )\frac{%
2\ell +1}{4\pi }=\lim_{j\rightarrow \infty }B^{j(\gamma -2)}B^{j}\sum_{\ell
=1}^{\infty }\int_{\frac{\ell }{B^{j}}}^{\frac{\ell +1}{B^{j}}%
}dx\;b^{2}\left( \frac{\ell }{B^{j}};p\right) \ell ^{-\gamma }G(\ell )\frac{%
2\ell +1}{4\pi }  \notag  \label{c_0} \\
& =\lim_{j\rightarrow \infty }B^{j(\gamma -2)}B^{j}\sum_{\ell =1}^{\infty
}\int_{\frac{\ell }{B^{j}}}^{\frac{\ell +1}{B^{j}}}dx\;b^{2}(\frac{\lfloor
B^{j}x\rfloor }{B^{j}};p)\lfloor B^{j}x\rfloor ^{-\gamma }G(\lfloor
B^{j}x\rfloor )\frac{2\lfloor B^{j}x\rfloor +1}{4\pi }  \notag \\
& =\lim_{j\rightarrow \infty }B^{j(\gamma -2)}B^{j}\int_{\frac{1}{B^{j}}%
}^{\infty }\;b^{2}(\frac{\lfloor B^{j}x\rfloor }{B^{j}};p)\lfloor
B^{j}x\rfloor ^{-\gamma }G(\lfloor B^{j}x\rfloor )\frac{2\lfloor
B^{j}x\rfloor +1}{4\pi }dx  \notag \\
& =\lim_{j\rightarrow \infty }\int_{\frac{1}{B^{j}}}^{\infty }\;b^{2}(\frac{%
\lfloor B^{j}x\rfloor }{B^{j}};p)\left( \frac{\lfloor B^{j}x\rfloor }{B^{j}}%
\right) ^{-\gamma }\frac{G(\lfloor B^{j}x\rfloor )}{2\pi }\frac{2\lfloor
B^{j}x\rfloor +1}{B^{j}2}dx  \notag \\
& =\frac{G_{0}}{2\pi }\int_{0}^{\infty }\;b^{2}(x;p)x^{1-\gamma
}dx=\frac{G_0}{2 \pi}2^{\gamma /2-2-2p}\Gamma (1-\gamma /2+2p)=c_{p,0}(\gamma )\text{ .}
\notag
\end{align}%
Likewise
\begin{equation*}
\begin{split}
c_{p,2n}(\gamma )& =\lim_{j\rightarrow \infty }B^{j(\gamma -2-2n)}\sum_{\ell
=1}^{\infty }b^{2}(\frac{\ell }{B^{j}};p)\ell ^{-\gamma }G(\ell )\frac{2\ell
+1}{4\pi }\ell ^{n}(\ell +1)^{n} \\
& =\lim_{j\rightarrow \infty }B^{j(\gamma -2-2n)}B^{j}\int_{\frac{1}{B^{j}}%
}^{\infty }\;b^{2}(\frac{\lfloor B^{j}x\rfloor }{B^{j}};p)\lfloor
B^{j}x\rfloor ^{-\gamma }G(\lfloor B^{j}x\rfloor )\frac{2\lfloor
B^{j}x\rfloor +1}{4\pi }\lfloor B^{j}x\rfloor ^{n}(\lfloor B^{j}x\rfloor
+1)^{n}dx \\
& =\lim_{j\rightarrow \infty }\int_{\frac{1}{B^{j}}}^{\infty }\;b^{2}(\frac{%
\lfloor B^{j}x\rfloor }{B^{j}};p)\left( \frac{\lfloor B^{j}x\rfloor }{B^{j}}%
\right) ^{-\gamma }\frac{G(\lfloor B^{j}x\rfloor )}{2\pi }\frac{2\lfloor
B^{j}x\rfloor +1}{B^{j}2}\left( \frac{\lfloor B^{j}x\rfloor }{B^{j}}\right)
^{n}(\frac{\lfloor B^{j}x\rfloor +1}{B^{j}})^{n}dx \\
& =\frac{G_{0}}{2\pi }\int_{0}^{\infty }\;b^{2}(x;p)x^{2n+1-\gamma }dx =\frac{G_0}{2 \pi}2^{\gamma /2-2-n-2p}\Gamma (1-\gamma /2+n+2p)\text{ .}
\end{split}
\end{equation*}
Hence
\begin{align*}
\lim_{j\rightarrow \infty }\mathcal{B}_{2n,p,j}& =\lim_{j\rightarrow \infty }%
\frac{B^{j(\gamma -2)}}{B^{j(\gamma -2)}\sum_{\ell =1}^{\infty }b^{2}(\frac{%
\ell }{B^{j}};p)C_{\ell }\frac{2\ell +1}{4\pi }}\frac{B^{j(\gamma
-2-2n)}\sum_{\ell =1}^{\infty }b^{2}(\frac{\ell }{B^{j}};p)C_{\ell }\frac{%
2\ell +1}{4\pi }\ell ^{n}(\ell +1)^{n}}{B^{j(\gamma -2-2n)}} \\
& =\lim_{j\rightarrow \infty }\frac{B^{j(\gamma -2)}}{c_{p,0}(\gamma )}\frac{%
c_{p,2n}(\gamma )}{B^{j(\gamma -2-2n)}}=\lim_{j\rightarrow \infty }\frac{1}{%
c_{p,0}(\gamma )}\frac{c_{p,2n}(\gamma )}{B^{-2jn}}=\lim_{j\rightarrow
\infty }\frac{c_{p,2n}(\gamma )}{c_{p,0}(\gamma )}B^{2jn}
\end{align*}%
and finally
\begin{equation} \label{limj}
\lim_{j\rightarrow \infty }B^{-2nj}\mathcal{B}_{2n,p,j}=\frac{%
c_{p,2n}(\gamma )}{c_{p,0}(\gamma )}=\frac{2^{\gamma /2-2-n-2p}\Gamma (1-\gamma /2+n+2p)}{2^{\gamma /2-2-2p}\Gamma
(1-\gamma /2+2p)}=2^{-n}\frac{\Gamma (1-\gamma /2+n+2p)}{
\Gamma (1-\gamma /2+2p)}.
\end{equation}
It follows that
\begin{equation}
\mathrm{Var}(\beta _{j})\sim \frac{c_{p,0}(\gamma )}{B^{j(\gamma -2)}}\text{
,}  \label{eq:var-beta}
\end{equation}%
and%
\begin{equation*}
\lim_{j\rightarrow \infty }B^{-2j}C^{\prime }(\tilde{\beta}_{j})=\frac{%
\int_{0}^{\infty }b^{2}(u;p)u^{3-\gamma }du}{2\int_{0}^{\infty
}b^{2}(u;p)u^{1-\gamma }du}=\frac{c_{p,2}(\gamma )}{2c_{p,0}(\gamma )}, \lim_{j\rightarrow \infty }B^{-4j}C^{\prime \prime }(\tilde{\beta}_{j})=%
\frac{\int_{0}^{\infty }b^{2}(u;p)u^{5-\gamma }du}{8\int_{0}^{\infty
}b^{2}(u;p)u^{1-\gamma }du}=\frac{c_{p,4}(\gamma )}{8c_{p,0}(\gamma )},
\end{equation*}%
yielding the desired results.
\end{proof}

In the sequel, we let $p = 1$; all the results below can be trivially extended to choices of other forms of filtering, with different values of $p$ (we stress that the case $p=1$ is the choice that has been usually adopted for applications, see \cite{scodeller,scodeller2,scodellermexican}).

\section{The multiple testing scheme \label{S:multipletesting}}

\subsection{The signal and null regions \label{subsec:tolerance}}

To properly approach the detection of point sources as a multiple testing problem, we first need to carefully define the spatial region occupied by the needlet-transformed point sources.

We recall that in Condition \ref{C:tN&Bj} we required that $t_{N}\rightarrow
0$ as $j\rightarrow \infty $, meaning that we work in a setting where
signals get more and more concentrated in the asymptotic limit; this is
clearly a necessary condition for meaningful results, as we are going to
handle an increasing number of signals (and tests) and, because we are
working on a compact domain, in the absence of such increasing localization
the problem would become entirely trivial (the signal region would cover the
sphere). At the same time, we required the kernel to
concentrate as well, again to avoid an excessive leakage of signal which
would make the whole approach meaningless.

To be more precise, define the \emph{signal region} $\mathbb{D}_{1}^{\rho
}=\cup _{k=1}^{N}D_{k}^{\rho }=\cup _{k=1}^{N}B(\xi _{k},\rho )$ and \emph{%
null region} $\mathbb{D}_{0}^{\rho }={\mathbb{S}}^{2}\setminus \mathbb{D}%
_{1}^{\rho }$, where $\rho >0$ is a pre-specified location tolerance
parameter and $D_{k}^{\rho }=B(\xi _{k},\rho )$ is the geodesic ball on ${%
\mathbb{S}}^{2}$ with center $\xi _{k}$ and radius $\rho $. The presence of
a tolerance parameter is not only required to settle properly the
theoretical framework, but is also consistent with the common scientific
practice, see again \cite{scodeller}, \cite{scodeller2}. We introduce now a
further condition.

\begin{condition}\label{tolerance}
As $j\rightarrow \infty ,$ we have%
\begin{equation*}
\begin{gathered}
\rho = \rho _{j} \sim j^\nu B^{-j} \\
\forall N, \min_{1\leq k\neq k^{\prime }\leq N}d(\xi _{k},\xi _{k^{\prime }}) >\rho, \\
\inf_{N}\inf_{1\leq k\leq N}a_{k} > a_0,
\end{gathered}
\end{equation*}
where $\nu$ and $a_0$ are positive constants.
\end{condition}

The first two lines of Condition \ref{tolerance} are meant to ensure that
the tolerance radius $\rho$ decays to zero asymptotically faster than the distance
between separate sources. Otherwise proper identification and counting of point sources would become unfeasible.

Notice that the restriction $\min_{1\leq k\neq k^{\prime }\leq N}d(\xi _{k},\xi _{k^{\prime }}) >\rho$ in Condition \ref{tolerance} yields
\[
\limsup_{j\to \infty}\mathrm{Area}(B(\xi_1,\rho_j ))N_j < 4\pi = \mathrm{Area}(\S^2),
\]
implying that the area of null region is always positive and that $N_j$ cannot grow too fast, specifically $N_j=O(\rho_j^{-2})=O(j^{-2\nu} B^{2j})$. It is easy to check that
\[
{\rm Area}(B(\xi_1, \rho_j)) = 2\pi(1-\cos\rho_j) \sim \pi \rho_j^2 \sim \pi j^{2\nu} B^{-2j}.
\]

Here are some examples for Condition \ref{tolerance}. If $N_j=B^{2j(1-\delta )}$ for some $0<\delta <1/2$, then the area of signal region tends to 0. If $N_j=c_{0}\rho_j^{-2}=c_{0}j^{-2\nu}B^{2j}$ for some $c_{0}\in (0,4)$, then the area of signal region tends to $\pi c_0$.

\subsection{The STEM algorithm on the sphere \label{subsec:STEM}}

As some general notation, for a smooth Gaussian random field $\{X(x),x\in \S^2\}$, define the number of local maxima of $X$ exceeding the level $u\in {\mathbb{R}}$ over a domain $D \subset {\mathbb{S}}^{2}$ as
\begin{equation}
M_{u}(X;D)=\#\left\{ x\in D:X(x)>u,\ \nabla X(x)=0,\ \nabla ^{2}X(x)\prec
0\right\} ;  \label{eq:LocalMax}
\end{equation}%
here $\nabla X(x)$ and $\nabla ^{2}X(x)$ denote the gradient and Hessian of the field $X$ at $x$, and $\nabla ^{2}X(x)\prec 0$ means the Hessian $\nabla ^{2}X(x)$ is
negative definite. The gradient and Hessian can be computed as $\nabla X = (E_1 X, E_2 X)$ and $\nabla^2 X = (E_iE_j X)_{1\le i, j\le 2}$, respectively, where $E_1$ and $E_2$ are orthonomal tangent vectors. In spherical coordinates $0\leq \theta \leq \pi$, $0\leq \varphi <2\pi $, these are given by
\[
E_1 = \frac{\partial}{\partial \theta}, \quad E_2 = \frac{1}{\sin \theta}\frac{\partial}{\partial \varphi}.
\]
For convenience, denote by $M(X;D)=M_{-\infty }(X;D)$ the
total number of local maxima of $X$ over $D$.

Suppose now we observe $y_{N}(t)$ on ${\mathbb{S}}^{2}$ defined by %
\eqref{eq:signal+noise}.  The STEM algorithm of \cite{Schwartzman:2011,chengschwartzman2} takes in our case the following form.

\begin{alg}[STEM algorithm]
\label{alg:STEM} \hfill

\noindent

\begin{enumerate}
\item \emph{Kernel smoothing}: Apply the needlet transform to the observed field
\eqref{eq:signal+noise} to obtain the filtered field \eqref{eq:conv}. Normalize by the (known) noise variance to obtain the field \eqref{eq:signal+noise:norm}.

\item \emph{Candidate peaks}: Find the set of local maxima of $y_{N,j}(x)$
on ${\mathbb{S}}^{2}$
\begin{equation}  \label{eq:Tv}
\tilde{T}_{N,j} = \left\{ x\in {\mathbb{S}}^{2}:\
\nabla \tilde{y}_{N,j}(x)=0,\ \nabla ^{2}\tilde{y}_{N,j}(x)\prec 0\right\}.
\end{equation}

\item \emph{P-values}: For each $x\in \tilde{T}_{N,j}$, compute the
p-value $p_{N,j}(x)$ for  testing
\begin{equation}  \label{eq:null-hypothesis}
\begin{aligned} \mathcal{H}_{0}(x)&: \ \{\mu_{N,j}(y) = 0 \text{ for all } y \in
B(x, \rho_j)\} \quad \text{\rm vs.} \\ \mathcal{H}_{A}(x)&: \ \{\mu_{N,j}(y) > 0
\text{ for some } y \in B(x, \rho_j)\} \end{aligned}
\end{equation}
where $B(x, \rho_j)$ is a geodesic ball centered at $x$ on the sphere and of radius equal to the tolerance radius $\rho_j$.

\item \emph{Multiple testing}: Notice that $M(\tilde{y}_{N,j}; {\mathbb{S}}%
^2)=\#\{x\in \tilde{T}_{N,j}\}$ is the number of tested hypotheses.
Perform a multiple testing procedure on the set of $M(\tilde{y}_{N,j}; {%
\mathbb{S}}^2)$ p-values $\{p_{N,j}(x),\,x\in \tilde{T}_{N,j}\}$, and
declare significant all local maxima whose p-values are smaller than the
significance threshold.
\end{enumerate}
\end{alg}

Next, we carefully define detection errors and power for this testing scheme.

\subsection{Error and power definitions \label{subsec:error}}

Now, for fixed $u\in {\mathbb{R}}$, denote by $\tilde{T}_{N,j}(u)$ the set of local maxima exceeding $u$ defined via \eqref{eq:LocalMax}. Define the total number of detected peaks and the number of falsely detected peaks as
\begin{equation}
R_j(u)=M_{u}(\tilde{y}_{N,j};{\mathbb{S}}^{2}),\quad V_{\rho_j}(u)=M_{u}(\tilde{%
y}_{N,j};\mathbb{D}_{0}^{\rho_j }),  \label{eq:R-and-V}
\end{equation}%
respectively. Both are defined as zero if $\tilde{T}_{N,j}(u)$, is empty.
As usual, the False Discovery Proportion (FDP) is proportion of falsely detected peaks, i.e.
\begin{equation}\label{eq:FDP}
\mathrm{FDP}_{\rho_j}(u) = \frac{V_{\rho_j}(u)}{R_j(u)\vee 1} =
\frac{M_u(\tilde{y}_{N,j}; {\mathbb{D}}_0^{\rho_j})}{M_u(%
\tilde{y}_{N,j}; {\mathbb{D}}_0^\rho)+M_u(\tilde{y}_{N,j}; {\mathbb{D}}%
_1^{\rho_j})} ,
\end{equation}%
while the False Discovery Rate (FDR) is the expected FDP, i.e.
\begin{equation}
\mathrm{FDR}_{\rho_j}(u)=\mathrm{E}\left\{ \frac{V_{\rho_j}(u)}{R_j(u)\vee 1}%
\right\} .  \label{eq:FDR}
\end{equation}%
We shall denote $W_{\rho_j}(u)=R_j(u)-V_{\rho_j}(u)$.

Finally, again following the same conventions as in \cite{Schwartzman:2011},
\cite{chengschwartzman2}, we define the power of Algorithm \ref{alg:STEM} as
the expected fraction of true discovered peaks
\begin{equation}
\mathrm{Power}_{\rho_j}(u)=\mathrm{E}\left( \frac{1}{N_j}\sum_{k=1}^{N_j}%
\mathbbm{1}_{\left\{ \tilde{T}_{N,j}(u)\cap D_{k}^{\rho }\neq \emptyset
\right\} }\right) =\frac{1}{N_j}\sum_{k=1}^{N_j}\mathrm{Power}_{\rho_j ,k}(u),
\label{eq:power}
\end{equation}%
where $\mathrm{Power}_{\rho_j ,k}(u)$ is the probability of detecting peak $k$
\begin{equation}
\mathrm{Power}_{\rho_j ,k}(u)={\mathbb{P}}\left( \tilde{T}_{N,j}(u)\cap
D_{k}^{\rho_j }\neq \emptyset \right) .  \label{eq:power-j}
\end{equation}%
The indicator function in \eqref{eq:power} ensures that only one significant
local maximum is counted within the same peak support, so power is not
inflated.

\subsection{P-values and BH algorithm \label{S:pvalues+BH}}

\label{sec:pvalue} Given the observed heights $y_{N,j}(x)$ at the local
maxima $x\in \tilde{T}_{N,j}$, the p-values in step (3) of Algorithm \ref%
{alg:STEM} are computed as $p_{N,j}(t)=F_j\left( y_{N,j}(t)\right) $, $%
t\in \tilde{T}_{N,j}$, where
\begin{equation}
F_j(u)={\mathbb{P}} \left(\tilde{ \beta} _{j}(x)>u~\Big|~x\in \tilde{T}%
_{N,j}\right)  \label{eq:palm}
\end{equation}%
denotes the right tail probability of $\tilde{\beta} _{j}(x)$ at the local
maximum $x\in \tilde{T}_{N,j}$, evaluated under the complete null
hypothesis $\mu (x)=0,\forall x$.

Applying the technique of \cite{chengschwartzman2015m}, we have that
\begin{equation}
F_{j}(u) = \frac{{\mathbb{E}}[M_{u}(\tilde{\beta} _{j}; {\mathbb{S}}^{2})]}
{{\mathbb{E}}[M(\tilde{\beta}_{j};{\mathbb{S}}^{2})]}
= \int_{u}^{\infty }f_j(x)\,dx,
\label{eq:height-distr}
\end{equation}%
where
\begin{equation*}
\begin{split}
f_j(x)& =\frac{2\sqrt{3+\kappa _{1,j}}}{2+\kappa _{1,j}\sqrt{3+\kappa
_{1,j}}}\Bigg\{\left[ \kappa _{1,j}+\kappa _{2,j}(x^{2}-1)\right] \phi
(x)\Phi \left( \frac{\sqrt{\kappa _{2,j}}x}{\sqrt{2+\kappa _{1,j}-\kappa
_{2,j}}}\right)  \\
& \quad +\frac{\sqrt{\kappa _{2,j}(2+\kappa _{1,j}-\kappa _{2,j})}}{2\pi }%
xe^{-\frac{(2+\kappa _{1,j})x^{2}}{2(2+\kappa _{1,j}-\kappa _{2,j})}} \\
& \quad +\frac{\sqrt{2}}{\sqrt{\pi (3+\kappa _{1,j}-\kappa _{2,j})}}e^{-%
\frac{(3+\kappa _{1,j})x^{2}}{2(3+\kappa _{1,j}-\kappa _{2,j})}}\Phi \left(
\frac{\sqrt{\kappa _{2,j}}x}{\sqrt{(2+\kappa _{1,j}-\kappa _{2,j})(3+\kappa_{1,j}-\kappa _{2,j})}}\right) \Bigg\}.
\end{split}%
\end{equation*}
Here $\kappa_{1,j}$, $\kappa_{2,j}$ are defined in \eqref{eq:kappa} above, and $\phi(x)$ and $\Phi(x)$ denote the standard normal density and distribution functions, respectively.

We can now apply the BH procedure in step (4) of Algorithm \ref{alg:STEM},
as follows. For a fixed significance level $\alpha \in (0,1)$, let $k$ be
the largest index for which the $i$th smallest $p$-value is less than $i\alpha /M(\tilde{y}_{N,j};{\mathbb{S}}^{2})$. Then the null hypothesis $%
\mathcal{H}_{0}(x)$ at $x\in \tilde{T}_{N,j}$ is rejected if
\begin{equation}
p_{N,j}(x)<\frac{k\alpha }{M(\tilde{y}_{N,j};{\mathbb{S}}^{2})}\quad
\iff \quad \tilde{y}_{N,j}(x)>\tilde{u}_{\mathrm{BH},j} = F_{j}^{-1}\left(
 \frac{k\alpha }{M(\tilde{y}_{N,j};{\mathbb{S}}^{2})}%
\right) ,  \label{eq:thresh-BH-random}
\end{equation}%
where $k\alpha / M(\tilde{y}_{N,j};{\mathbb{S}}^{2})$ is defined as 1 if $M(\tilde{y}_{N,j};{\mathbb{S}}^{2})=0$.
Since $\tilde{u}_{\mathrm{BH},j}$ is random, definition \eqref{eq:FDR} is
hereby modified to
\begin{equation}
\mathrm{FDR}_{\mathrm{BH},{\rho_j}}={\mathbb{E}}\left\{ \frac{V_{\rho_j}(%
\tilde{u}_{\mathrm{BH},j})}{R_j(\tilde{u}_{\mathrm{BH},j})\vee 1}%
\right\} ,  \label{eq:true-FDR}
\end{equation}%
where $R_j(\cdot )$ and $V_{\rho_j}(\cdot )$ are defined in \eqref{eq:R-and-V}
and the expectation is taken over all possible realizations of the random
threshold $\tilde{u}_{\mathrm{BH},j}$.

Since $\tilde{u}_{\mathrm{BH},j}$ is random, similarly to the definition
of $\mathrm{FDR}_{\mathrm{BH},\rho}$ \eqref{eq:true-FDR}, we define
\begin{equation}  \label{eq:true-power}
\mathrm{Power}_{\mathrm{BH}, {\rho_j}} = {\mathbb{E}} \left( \frac{1}{N_j}
\sum_{k=1}^{N_j} \mathbbm{1}_{\left\{ \tilde{T}_{N,j}(\tilde{u}_{\mathrm{BH}%
,j}) \cap D_k^\rho \ne \emptyset \right\}}\right).
\end{equation}

\section{FDR Control \label{S:proofFDR} and Power Consistency}

\subsection{FDR Control}

The strategy to prove FDR control is to first quantify the expected number of local maxima above any level $u$ over the null region (false discoveries) and signal region (true discoveries). This is given in Lemmas \ref{Lem:null-region} and \ref{Lem:signal-region} below.

\begin{lemma}
\label{Lem:null-region} Let $u\in {\mathbb{R}}$ be fixed. Then as $j\to
\infty$, the expected number of local maxima of $\tilde{y}_{N,j}$ above $u$ in the null region $\mathbb{D}_0^{\rho_j}$ is
\begin{equation} \label{eq:local-max-null}
{\mathbb{E}}[M_u(\tilde{y}_{N,j}; {\mathbb{D}}_0^{\rho_j})] = [4\pi -2\pi(1-\cos\rho_j)N_j]r_j(u)+o(e^{-j^{\nu}}),
\end{equation}
where
\begin{equation} \label{eq:r(u)}
r_j(u) = {\mathbb{E}}[M_u(\tilde{\beta}_{N,j}; {\mathbb{S}}^{2})] = F_{j}(u) r_{j},
\quad u\in {\mathbb{R}},
\end{equation}
is the expected number of local maxima of $\tilde{\beta}_{j}$ exceeding $u$
over a unit area on ${\mathbb{S}}^2$, $F_j(u)$ is the tail distribution function \eqref{eq:height-distr}, and
\begin{equation}\label{rj}
r_{j}=r_j(-\infty)=\frac{1}{4\pi }+\frac{1}{2\pi \kappa _{1,j}\sqrt{3+\kappa _{1,j}}}.
\end{equation}%
\end{lemma}

\begin{proof}\ Recall ${\rm Area}(B(\xi_k, \rho_j)) = 2\pi(1-\cos\rho_j)$ for every $k$, therefore
\[
{\rm Area}(\mathbb{D}_0^{\rho_j}) ={\rm Area}(\S^2) - {\rm Area}(B(\xi_k, \rho_j)) N_j = 4\pi -2\pi(1-\cos\rho_j)N_j.
\]	
By the Kac-Rice metatheorem, Lemma \ref{GellerMayeli} and Condition \ref{tolerance},
\begin{equation*}
\begin{split}
\E&[M_{u}(\tilde{y}_{N,j};\mathbb{D}_{0}^{\rho_j })]\\
&=\int_{\mathbb{D}_0^{\rho_j}}\frac{1}{2\pi \sqrt{{\rm det Cov}(\nabla\tilde{y}_{N,j}(x))}}{\mathbb{E}}[|\text{det}(\nabla ^{2}%
\tilde{y}_{N,j}(x))|\mathbbm{1}_{\{\tilde{y}_{N,j}(x)>u,\,\nabla ^{2}\tilde{y%
}_{N,j}(x)\prec 0\}}|\nabla \tilde{y}_{N,j}(x)=0]dx \\
& =\int_{\mathbb{D}_0^{\rho_j}}\frac{1}{2\pi \sqrt{{\rm det Cov}(\nabla\tilde{\beta}_{j}(x))}}{\mathbb{E}}[|\text{det}(\nabla ^{2}%
\tilde{\beta}_{j}(x))|\mathbbm{1}_{\{\tilde{\beta}_{j}(x)>u,\,\nabla ^{2}%
\tilde{\beta}_{j}(x)\prec 0\}}|\nabla \tilde{\beta}_{j}(x)=0]dx \\
& \quad +O(N_jB^{mj}j^{2\nu}e^{-j^{2\nu}}),
\end{split}%
\end{equation*}%
where $m$ is some positive constant. Evaluating the integral yields \eqref{eq:local-max-null}, where
\begin{equation*}
r_j(u) = \frac{1}{2\pi C'(\tilde{\beta}_j)}{\mathbb{E}}[ |\mathrm{det} (\nabla^2 \tilde{\beta}_{j}(x))|%
\mathbbm{1}_{\{\tilde{\beta}_{j}(x)>u,\, \nabla^2 \tilde{\beta}_{j}(x) \prec
0\}} | \nabla \tilde{\beta}_{j}(x)=0]
\end{equation*}
is the expected number of local maxima of $\tilde{\beta}_{j}$ exceeding $u$
over a unit area on ${\mathbb{S}}^2$. The exact expression \eqref{eq:r(u)} follows from \eqref{eq:height-distr}, while \eqref{rj} was proved in \cite{chengschwartzman2015m}.
\end{proof}

\begin{remark} \label{remark:LocalMax} \textbf{[Asymptotics of $r_j(u)$.]}
By Proposition \ref{prop:kappa}, as $j\rightarrow \infty $,
\begin{equation*}
\kappa _{1,j}\sim \frac{4c_{p,2}(\gamma )}{c_{p,4}(\gamma )}B^{-2j},
\end{equation*}%
implying for \eqref{rj} and \eqref{eq:r(u)} that
\begin{equation*}\label{rju}
r_{j}\sim \frac{c_{p,4}(\gamma )}{8\pi \sqrt{3}c_{p,2}(\gamma )}B^{2j},
\qquad r_{j}(u)=F_{j}(u)r_{j}\sim F_{j}(u)\frac{c_{p,4}(\gamma )}{8\pi \sqrt{3}%
c_{p,2}(\gamma )}B^{2j}.
\end{equation*}
\end{remark}

\begin{lemma}
\label{Lem:signal-region} Let $u\in {\mathbb{R}}$ be fixed. Then as $j\to \infty$, the number of local maxima of $\tilde{y}_{N,j}$ over the signal region $\mathbb{D}_1^{\rho_j}$ satisfies
\begin{equation*}
\begin{aligned}
M_u(\tilde{y}_{N,j}; {\mathbb{D}}_1^{\rho_j})&\ge N_j+O_p(B^{-2j}) \\
{\mathbb{E}}[M_u(\tilde{y}_{N,j}; {\mathbb{D}}_1^{\rho_j})] &\ge N_j+O(B^{-2j}).
\end{aligned}
\end{equation*}
\end{lemma}

\begin{proof}\ Let $\tilde{\rho}_j=B^{-j} < \rho_j = j^\nu B^{-j}$. By Lemma \ref{GellerMayeli}, within the domain $B(\xi_k, \tilde{\rho}_j)$, the mean function of $\tilde{y}_{N,j}$ satisfies the assumptions of the unimodal signal model in \cite{chengschwartzman2} with the signal strength being $a=B^{2j}$. It then follows from similar arguments as in \cite{chengschwartzman2} that
\begin{equation*}
\begin{aligned}
M_u(\tilde{y}_{N,j}; {\mathbb{D}}_1^{\tilde{\rho}_j}) &= N_j+O_p(B^{-2j}) \\
{\mathbb{E}}[M_u(\tilde{y}_{N,j}; {\mathbb{D}}_1^{\tilde{\rho}_j})] &= N_j+O(B^{-2j}).
\end{aligned}
\end{equation*}
The desired results then follow immediately from the observation
\begin{equation*}
M_u(\tilde{y}_{N,j}; {\mathbb{D}}_1^{\rho_j}) \ge M_u(\tilde{y}_{N,j}; {\mathbb{D}}_1^{\tilde{\rho}}),
\end{equation*}
where the inequality admits the possibility of there being other local maxima in the flatter areas $B(\xi_k, \rho_j)\setminus B(\xi_k, \tilde{\rho}_j)$ of the needlet transform impulse response. The exact expected number of these is presumably small, but hard to estimate.
\end{proof}

\begin{remark} \textbf{[The rate of $\rho_j$]}
The proof of Lemma \ref{Lem:signal-region} explains why the we make the assumption $\rho _{j}\sim j^\nu B^{-j}$ in Condition \ref{tolerance} above. This choice of rate for $\rho_j$, decaying slightly less slowly than $B^{-j}$, allows obtaining an asymptotic limit to the number of local maxima over the null region $\E[M_u(\tilde{y}_{N,j}, \mathbb{D}_{0}^{\rho_j})]$, while the number of local maxima over the signal region $\E[M_u(\tilde{y}_{N,j}, \mathbb{D}_{1}^{\rho_j})]$ can be bounded asymptotically. If we had chosen the rate of $\tilde{\rho}_j$ for $\rho_j$, decaying at a rate $B^{-j}$, then as shown in the proof of Lemma \ref{Lem:signal-region}, we could obtain an exact limit for the number of local maxima over the signal region; however in that case, the number of local maxima over the null region would be difficult to quantify due to complicated behavior of the mean function $\tilde{\mu}_{N,j}$ (after the needlet tranform) immediately outside that radius.
\end{remark}

The following ergodic result, which will be used in the proof of Theorem \ref{thm:FDR} below, shows that the variance of the number of local maxima goes to zero after normalization by the expected value. The result itself is theoretically important and the proof is given in the Appendix.

\begin{theorem} \label{var_noise}
	\label{var} As $j\rightarrow \infty $
	\begin{equation*}
	\mathrm{Var}[M_{u}(\tilde{\beta}_{j};{\mathbb{S}}%
	^{2})]\le c(u)j^{2}B^{2j}+o(j^{2}B^{2j}).
	\end{equation*}%
	where the constant $c(u)$ is uniformly bounded with respect to $u$ and the $%
	o(\cdot )$ term is universal.
\end{theorem}

Following is the first main result of this paper, showing control of FDP and FDR.

\begin{theorem}\label{thm:FDR}
Let the assumptions in the model and Condition \ref{tolerance} hold.

(i) Suppose that Algorithm \ref{alg:STEM} is applied with a fixed threshold $u$, then
\begin{equation}  \label{eq:bound-FDR-u}
\mathrm{FDP}_{\rho_j} \le \frac{[4\pi -2\pi(1-\cos\rho_j)N_j]r_j(u)}{[4\pi -2\pi(1-\cos\rho_j)N_j]r_j(u)+N_j}(1+o_p(1)),
\end{equation}
where $r_j(u)$ is defined in \eqref{eq:r(u)}.

(ii) Suppose that Algorithm \ref{alg:STEM} is applied with the random
threshold $\tilde{u}_{\mathrm{BH},j}$ \eqref{eq:thresh-BH-random}, then
\begin{equation}  \label{eq:bound-FDR}
\mathrm{FDR}_{\mathrm{BH},\rho_j} \le \alpha\frac{[4\pi -2\pi(1-\cos\rho_j)N_j]r_j}{%
[4\pi -2\pi(1-\cos\rho_j)N_j]r_j+N_j}+o(1),
\end{equation}
where $r_j$ is given by \eqref{rj}.
\end{theorem}

\begin{proof}\par
(i) By Theorem \ref{var_noise} and Chebyshev's inequality,
\begin{equation*}
\begin{split}
\mathrm{FDP}_{\rho_j }(u)& =\frac{M_{u}(\tilde{y}_{N,j};{\mathbb{D}}_{0}^{\rho_j
})/B^{2j}}{M_{u}(\tilde{y}_{N,j};{\mathbb{D}}_{0}^{\rho_j })/B^{2j}+M_{u}(
\tilde{y}_{N,j};{\mathbb{D}}_{1}^{\rho_j })/B^{2j}} \\
& \leq \frac{{\mathbb{E}}[M_{u}(\tilde{y}_{N,j};{\mathbb{D}}_{0}^{\rho_j
})]/B^{2j}}{{\mathbb{E}}[M_{u}(\tilde{y}_{N,j};{\mathbb{D}}_{0}^{\rho_j
})]/B^{2j}+M_{u}(\tilde{y}_{N,j};{\mathbb{D}}_{1}^{\rho_j })/B^{2j}}(1+o_{p}(1)).
\end{split}
\end{equation*}
It then follows from Lemmas \ref{Lem:null-region} and \ref{Lem:signal-region} that
\begin{equation*}
\begin{split}
\mathrm{FDP}_{\rho_j }(u)=\frac{[4\pi -2\pi(1-\cos\rho_j)N_j]r_{j}(u)}{[4\pi -2\pi(1-\cos\rho_j)N_j]r_{j}(u)+N_j}(1+o_{p}(1)).
\end{split}
\end{equation*}

(ii) Following a similar argument to that in \cite{chengschwartzman2}, we use the fact that $\tilde{u}_{\mathrm{BH},j}$ is the smallest $u$ satisfying $\alpha \tilde{G}_{N,j}(u) \ge F_j(u)$, where
\begin{equation*}
\tilde{G}_{N,j}(u) = \frac{M_u(\tilde{y}_{N,j}; {\mathbb{D}}_0^{\rho_j}) + M_u(%
\tilde{y}_{N,j}; {\mathbb{D}}_1^{\rho_j})}{M(\tilde{y}_{N,j}; {\mathbb{D}}%
_0^{\rho_j}) + M(\tilde{y}_{N,j}; {\mathbb{D}}_1^{\rho_j})}
\end{equation*}
and $F_j(u)$ is the height distribution \eqref{eq:height-distr} of $\tilde{\beta}_j$.
Notice that
\begin{equation*}
F_j(u) = \frac{{\mathbb{E}}[M_u(\tilde{y}_{N,j}; {\mathbb{D}}_0^{\rho_j})]}{{%
\mathbb{E}}[M(\tilde{y}_{N,j}; {\mathbb{D}}_0^{\rho_j})]}.
\end{equation*}
Similarly to the proof of part (i), we have
\begin{equation*}
\begin{split}
\tilde{G}_{N,j}(u) &\ge \frac{{\mathbb{E}}[M(\tilde{y}_{N,j}; {\mathbb{D}%
}_0^{\rho_j})]F_j(u) + N_j}{{\mathbb{E}}[M(\tilde{y}_{N,j}; {\mathbb{D}}%
_0^{\rho_j})] + N_j} + o_p(1) \\
&=\frac{[4\pi -2\pi(1-\cos\rho_j)N_j]r_jF_j(u) + N_j}{[4\pi -2\pi(1-\cos\rho_j)N_j]r_j + N_j} + o_p(1).
\end{split}%
\end{equation*}
Solving the equation
\begin{equation*}
\alpha \frac{[4\pi -2\pi(1-\cos\rho_j)N_j]r_jF_j(u) + N_j}{[4\pi -2\pi(1-\cos\rho_j)N_j]r_j + N_j} +
o_p(1) = F_j(u)
\end{equation*}
gives an asymptotic solution
\begin{equation}  \label{eq:u*}
\tilde{u}_{\mathrm{BH}, j}^* = F_j^{-1}\left(\frac{\alpha N_j}{ N_j
+ (1-\alpha)[4\pi -2\pi(1-\cos\rho_j)N_j]r_j}\right) + o_p(1).
\end{equation}
Since $\tilde{u}_{\mathrm{BH},j} \le \tilde{u}_{\mathrm{BH}, j}^*$
almost surely, we have
\begin{equation*}
\begin{split}
\mathrm{FDR}_{\mathrm{BH}, j} &= {\mathbb{E}}\left[ \frac{V_{\rho_j}(\tilde{u}%
_{\mathrm{BH}, j})}{V_{\rho_j}(\tilde{u}_{\mathrm{BH}, j}) + W_{\rho_j}(\tilde{%
u}_{\mathrm{BH}, j})} \right]
\le {\mathbb{E}}\left[ \frac{V_{\rho_j}(\tilde{u}_{\mathrm{BH}, j}^*)}{%
V_{\rho_j}(\tilde{u}_{\mathrm{BH}, j}^*) + W_{\rho_j}(\tilde{u}_{\mathrm{BH},
j}^*)} \right] \\
&\le \frac{{\mathbb{E}}[V_{\rho_j}(\tilde{u}_{\mathrm{BH}, j}^*)]}{{\mathbb{E}%
}[V_{\rho_j}(\tilde{u}_{\mathrm{BH}, j}^*)] + {\mathbb{E}}[W_{\rho_j}(\tilde{u}_{%
\mathrm{BH}, j}^*)]} (1+o(1))
\le \frac{{\mathbb{E}}[M(\tilde{y}_{N,j}; {\mathbb{D}}_0^{\rho_j})]F_j(\tilde{u
}_{\mathrm{BH}, j}^*)}{{\mathbb{E}}[M(\tilde{y}_{N,j}; {\mathbb{D}}%
_0^{\rho_j})]F_j(\tilde{u}_{\mathrm{BH}, j}^*) + N_j} (1+o(1)) \\
&= \alpha\frac{[4\pi -2\pi(1-\cos\rho_j)N_j]r_j}{[4\pi -2\pi(1-\cos\rho_j)N_j]r_j+N_j}+o(1).
\end{split}%
\end{equation*}
\end{proof}

\begin{remark} \textbf{[Threshold for FDP]}
To make FDP asymptotically equal to a significance level $\alpha $, the corresponding threshold $u$ must satisfy the equation
\begin{equation*}
\frac{[4\pi -2\pi(1-\cos\rho_j)N_j]r_{j}(u)}{[4\pi -2\pi(1-\cos\rho_j)N_j]r_{j}(u)+N_j}=\alpha ,
\end{equation*}%
implying
\begin{equation}\label{eq:rj}
r_{j}(u)=\frac{\alpha N_j}{(1-\alpha )[4\pi -2\pi(1-\cos\rho_j)N_j]}.
\end{equation}%
By Remark \ref{remark:LocalMax},
\begin{equation*}
r_{j}(u)\sim F_{j}(u)\frac{c_{p,4}(\gamma )}{8\pi \sqrt{3}c_{p,2}(\gamma )}%
B^{2j},
\end{equation*}
implying that as $j\to \infty$ and $u\to \infty$,
\[
\log (r_{j}(u)) \sim \log (B^{2j}e^{-u^{2}/2}).
\]
Solving the equation
\begin{equation*}
B^{2j}e^{-u^{2}/2}=\frac{\alpha N_j}{(1-\alpha )[4\pi -2\pi(1-\cos\rho_j)N_j]}
\end{equation*}%
yields the approximate solution
\begin{equation}
u\sim  \sqrt{2\log (B^{2j}/N_j)}.
\label{eq:asymptotic-BH-threshold}
\end{equation}
According to Condition \ref{tolerance}, $N_j\rho_j ^{2}=O(1)$, implying $N_j=O(\rho_j
^{-2})=O(j^{-2}B^{2j})$. Therefore, $u\geq \sqrt{\log (j)}\rightarrow \infty
$.
\end{remark}

\begin{remark}
\textbf{[Comparison between FDP and BH Procedure.]} Dividing both sides of \eqref{eq:rj} by $r_{j}$ yields
\begin{equation*}
F_{j}(u)=\frac{\alpha N_j}{(1-\alpha )[4\pi -2\pi(1-\cos\rho_j)N_j]r_{j}},
\end{equation*}%
implying the following threshold by FDP for controlling significance level $\alpha $:
\begin{equation*}
u_{\alpha }=F_{j}^{-1}\left( \frac{\alpha N_j}{(1-\alpha )[4\pi -2\pi(1-\cos\rho_j)N_j]r_{j}}\right) .
\end{equation*}%
In comparison, for controlling significance level $\alpha$ by the BH procedure, the asymptotic threshold is given by \eqref{eq:u*}. If we replace $\alpha $ by
\begin{equation*}
\tilde{\alpha}=\alpha \frac{[4\pi -2\pi(1-\cos\rho_j)N_j]r_{j}}{[4\pi -2\pi(1-\cos\rho_j)N_j]r_{j}+N_j},
\end{equation*}%
then the FDP threshold at significance level $\tilde{\alpha}$ is given by
\begin{equation*}
u_{\tilde{\alpha}}=F_{j}^{-1}\left( \frac{\alpha N_j}{N_j+(1-\alpha
)[4\pi -2\pi(1-\cos\rho_j)N_j]r_{j}}\right).
\end{equation*}%
This coincides with the asymptotic threshold $\tilde{u}_{\mathrm{BH}%
,j}^{\ast }$ \eqref{eq:u*} by BH procedure. Since $r_{j}=O(B^{2j})$
and $N_j=O(\rho_j ^{-2})=O(j^{-2}B^{2j})$, we see that the upper bound in \eqref{eq:bound-FDR} tends to $\alpha $ in the limit of
high-frequency.
\end{remark}

\begin{remark} \textbf{[Comparison with FWER control (expected Euler characteristic).]}
For high values of the threshold $u$, the expected Euler characteristic
exceeding $u$ can be approximated by $r_{j}(u)$; hence, the threshold for
controlling the FWER can be obtained by solving the equation $%
r_{j}(u)=\alpha $. By the discussion in the previous remark, this equation becomes
$B^{2j}e^{-u^{2}/2}=\alpha$, which gives the solution
\begin{equation*}
u\sim \sqrt{2\log (B^{2j})} = 2\sqrt{j \log(B)}.
\end{equation*}
In comparison, the FDR threshold increases at a rate $\sqrt{\log(j)}$, much slower.
\end{remark}

\subsection{Power Consistency \label{S:powerconsistency}}
To prove power consistency, we first show that, asymptotically, there will be at least one local maximum of $\tilde{y}_{N,j}$ within a small ball centered at every point source.
\begin{lemma}
\label{Lem:power} For each fixed $k$, there exists $c>0$ such that for
sufficiently large $j$,
\begin{equation*}
\begin{split}
{\mathbb{P}}\left(\#\{x \in \tilde{T}_{N,j}(u)\cap D_k^{\tilde{\rho}}\} \ge 1
\right) &\ge 1-\exp\left(-c B^{j(\gamma-2)/2} \right),
\end{split}%
\end{equation*}
where $\tilde{\rho} = B^{-j}$ and $u=\sqrt{2\log (B^{2j}/N_j)}$ is the asymptotic BH threshold \eqref{eq:asymptotic-BH-threshold}.
\end{lemma}

\begin{proof}\ For each $k$, the probability $\tilde{y}_{N,j}(x)$ has at least one local maximum above $u$ in $D_{k}^{\tilde{\rho}}=B(\xi _{k},\tilde{\rho})$ is the complement of the probability that: (1) $\tilde{y}_{N,j}(x)$ has no local maxima in $%
D_{k}^{\tilde{\rho}}$, or (2) $\tilde{y}_{N,j}(x)$ is below $u$ everywhere in $D_{k}^{\tilde{\rho}}$.

For (1), 
this is less than the probability that there exists some $x\in D_{k}^{\rho _{1}}$
such that ${\langle }\nabla \tilde{y}_{N,j}(x),\xi _{k}-x\rangle \leq 0$,
since all $x\in \partial D_{k}^{\tilde{\rho}}$ satisfying ${\langle }\nabla
\tilde{y}_{N,j}(x),\xi _{k}-x\rangle >0$ would imply the existence of at
least one local maximum in $D_{k}^{\tilde{\rho}}$. This probability is bounded
above by
\begin{equation*}
\begin{aligned} {\mathbb P}&\left(\inf_{\partial D_k^{\tilde{\rho}}} {\langle}
\nabla \tilde{\beta}_j(x), \xi_k-x\rangle \le -\inf_{\partial D_k^{\tilde{\rho}}}
{\langle} \nabla \tilde{\mu}_{N,j}(x), \xi_k-x\rangle \right)\\ &= {\mathbb
P}\left(\sup_{\partial D_k^{\tilde{\rho}}} - \bigg\langle \nabla
\tilde{\beta}_j(x), \frac{\xi_k-x}{\|\xi_k-x\|} \bigg\rangle \ge
\inf_{\partial D_k^{\tilde{\rho}}} \bigg\langle \nabla \tilde{\mu}_{N,j}(x),
\frac{\xi_k-x}{\|\xi_k-x\|} \bigg\rangle \right)\\ &\le {\mathbb
P}\left(\sup_{x\in D_k^{\tilde{\rho}}}\sup_{\|\tau\|=1} {\langle} \nabla
\tilde{\beta}_j(x), \tau \rangle \ge c_1B^{2j+j(\gamma-2)/2} \right),
\end{aligned}
\end{equation*}%
where $c_{1}$ is a positive constant and the last inequality is due to %
\eqref{eq:var-beta}, \eqref{eq:lp-signal}, Lemma \ref{GellerMayeli} and the
fact that $\partial D_{k}^{\tilde{\rho}}$ is contained in the closure of $%
D_{k}^{\tilde{\rho}}$. By Proposition \ref{prop:kappa}, there exists $c_{2}>0$
such that for sufficiently large $j$,
\begin{equation*}
\sup_{x\in D_{k}^{\tilde{\rho}}}\sup_{\Vert \tau \Vert =1}\mathrm{Var}({\langle
}\nabla \tilde{\beta}_{j}(x),\tau \rangle )\leq c_{2}B^{2j}.
\end{equation*}%
Then by the Borell-TIS inequality, there exists $c_{3}>0$ such that for
sufficiently large $j$,
\begin{equation*}
{\mathbb{P}}\left( \#\{x\in \tilde{T}_{N,j}\cap D_{k}^{\tilde{\rho}}\}=0\right) \leq \exp \left( -c_{3}B^{j(\gamma -2)/2}\right).
\label{eq:I-mode-bound-0}
\end{equation*}

On the other hand, for (2), the probability that $\tilde{y}_{N,j}(x)$ is below $u$ everywhere in $D_{k}^{\tilde{\rho}}$ is bounded above by $1-\Phi(|u-B^{2j+j(\gamma-2)/2}|)$. The desired result then follows from the observation
\begin{equation*}
\begin{split}
{\mathbb{P}}\left(\#\{x \in \tilde{T}_{N,j}(u)\cap D_k^{\tilde{\rho}}\} \ge 1
\right)
\geq 1 - \exp \left( -c_{3}B^{j(\gamma -2)/2}\right) - \left(1-\Phi(|u-B^{2j+j(\gamma-2)/2}|)\right),
\end{split}%
\end{equation*}
where the last term in parentheses is much smaller than the second when $u=\sqrt{2\log (B^{2j}/N_j)}$.
\end{proof}

Following is the second main result of this paper, showing that the detection power tends to one asymptotically.

\begin{theorem}
\label{thm:power} Let the assumptions in the model and Condition \ref{tolerance} hold.

(i) Suppose that Algorithm \ref{alg:STEM} is applied with a fixed threshold $%
u$, then
\begin{equation*}
\mathrm{Power}_{\rho_j}(u) \to 1.
\end{equation*}

(ii) Suppose that Algorithm \ref{alg:STEM} is applied with the random
threshold $\tilde{u}_{\mathrm{BH}}$ \eqref{eq:thresh-BH-random}, then
\begin{equation*}
\mathrm{Power}_{\mathrm{BH}, \rho_j}\to 1.
\end{equation*}
\end{theorem}

\begin{proof}\
The desired results follow directly from Lemma \ref{Lem:power} and the definitions of power \eqref{eq:power} and \eqref{eq:true-power}.
\end{proof}

\section{Numerical Validation} \label{numerical}

In this section we present numerical evidence on the performance
of the algorithm advocated in this work.
One crucial step in the STEM algorithm (Algorithm \ref{alg:STEM}) is the computation of p-values of detected peaks, based on the distribution of peak heights under the complete noise assumption. We therefore start our
validation by comparing the analytical peak height distribution function
given in \refeq{eq:height-distr} with the empirical result from filtered noise
Monte Carlo simulations.

Once we establish the validity of the peak height distribution on the noise
field, we add simulated point sources to form the full signal-plus-noise Monte Carlo simulations. These simulations are used to evaluate the numerical performance of the asymptotic FDP approximation and FDR control of Section \ref{S:proofFDR}.

\subsection{Simulation of the CMB noise field}

All our maps and the corresponding spherical harmonic coefficients are
generated using the \emph{\healpix} package, which is now the standard
routine software for handling cosmological data: see \cite{healpix} for a
detailed discussion on this package and its main features. In \healpix one can use the \emph{create\_alm} routine to generate random spherical harmonic coefficients, $a_{\ell m}$, with a given power spectrum.  The code \emph{alm2map} takes these
coefficients and generate a pixelized Gaussian map; the inverse
process is implemented using the \emph{map2alm} code. To decompose a
map into Mexican needlet components, we filter the $a_{\ell m}$
coefficients by the Mexican needlet window functions as given in
\refeq{eq:betaj}.

A single \healpix pixel has an area of $4\pi/N_{\rm pix}$ where
$N_{\rm pix} = 12N_{\rm side}^2$ is the total number of pixels on a given map. The
resolution is specified by the $N_{\rm side}$ parameter, which is a multiple of 2.

To simulate our noise field, we generated 100 Gaussian realization
maps of the CMB sky starting from the Planck CMB power spectrum. All
maps are simulated with a pixel resolution of $N_{\rm side}=1024$.
The standard deviation, also called root mean square (RMS), of the simulated noise field is given by
\begin{equation}
  \sigma^2_{noise} = \sigma^2_{cmb} = \sum_\ell{\frac{(2\ell+1)C_\ell}{4\pi}},
  \label{eq:RMS}
\end{equation}
where $C_\ell$ is the Planck CMB power spectrum \cite{PlanckPS}.

To simulate the finite resolution of the measuring instrument, these maps are then
smoothed by a Gaussian filter with full-width half max (FWHM) of 10 arcmin. In the literature, this is usually referred to as a 10 arcmin Gaussian beam. Its effect can be thought of as part of the noise autocovariance function, although it is essentially negligible at $\ell\sim 1000$ as in our needlet analysis.

\subsection{Simulation of point sources}

As mentioned above, a point source in the sky is observed by a detector which has a finite angular resolution. With some abuse of nomenclature, the opening
angle of the smallest resolvable angular unit, $t_N$, is called the beam of the detector. The typical angular size of galaxies is of a few arcsecs (i.e., one degree divided by $60^2$) while the detector beam sizes for typical CMB experiments (10 arcmin) are an order of magnitude larger. This means that galaxies and other objects with
angular size smaller than the beam can be viewed as point sources.  As
argued in the previous sections, the convolution of the point sources
by the detector beam yields a Gaussian bell-like profile in the final
map with the peak of the Gaussian being at the location of the point
sources, $\xi_k$. The signal part of our simulations is hence given by
equation \refeq{eq:mu} above, where the coefficient $a_k$ represent the brightness of the $k$th point source and $N$ is their total number.

It would be possible to consider more realistic models for these point
sources, for instance using the so-called Planck sky model (see i.e.,
\cite{psm}). However, this would require a rather lengthy technical
discussion on some specific astrophysical and experimental settings,
which would not add anything substantial to the understanding of our
current algorithm, nor would alter significantly our numerical
results. We therefore delay a more complete analysis of these
practical issues to a future, more applied paper.

To simulate our signal model with $N$ point sources, we first
generated $N$ coordinate points randomly with a uniform probability
density over the sphere. Second we found the pixels that correspond to
these locations on the HEALpix map; third we set the values $a_k$ of these pixels as draws from a uniform distribution in the range 0 and $A_{\rm max}$.
These amplitudes are given as a multiple of the RMS of the noise \eqref{eq:RMS}.
Finally, to simulate the instrumental resolution, we convolved the map obtained in the last step by means of a Gaussian beam of ${\rm FWHM}=10$ arcmin.  This final map is now a pixelized version of \eqref{eq:signal+noise}. Notice that for
clarity we have described the smoothing process as a separate operation in
the noise and signal maps, but this is, of course, equivalent to doing
a single smoothing operation on a signal plus noise map.

The Gaussian beam decreases the sources
magnitude by a factor proportional to the ratio between the area of a
pixel and the area covered by the detector beam.
For our choice of the beam and the pixel resolution, this factor is an order
of magnitude. Since we desired to generate point sources uniformly
distributed between 0 and $3\sigma_{noise}$ after smoothing, we used
$A_{max}=30\sigma_{cmb}$. We considered different values for the
total number of sources, i.e., $N$=1000, 3000, 5000.

The final signal-plus-noise Monte Carlo simulations are then obtained
by adding the point sources map to the 100 noise simulations; an
example is provided in \reffig{fig:maps}.
Note that the point sources are weak and hard to find without statistical analysis.





\begin{figure}[H]
\begin{center}
  \includegraphics[width=0.98\textwidth,angle=0]{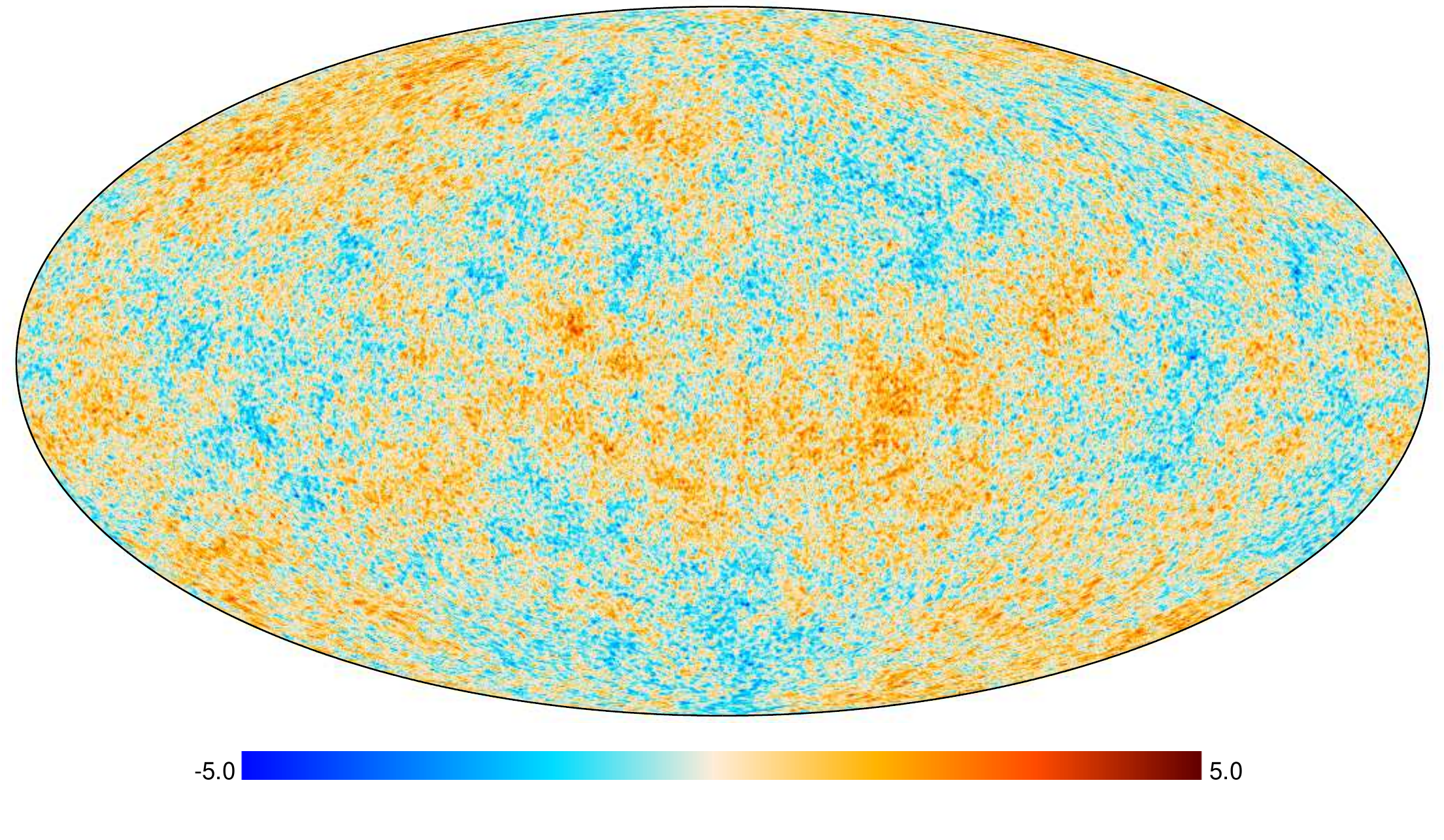}\\
  \includegraphics[width=0.49\textwidth,angle=0]{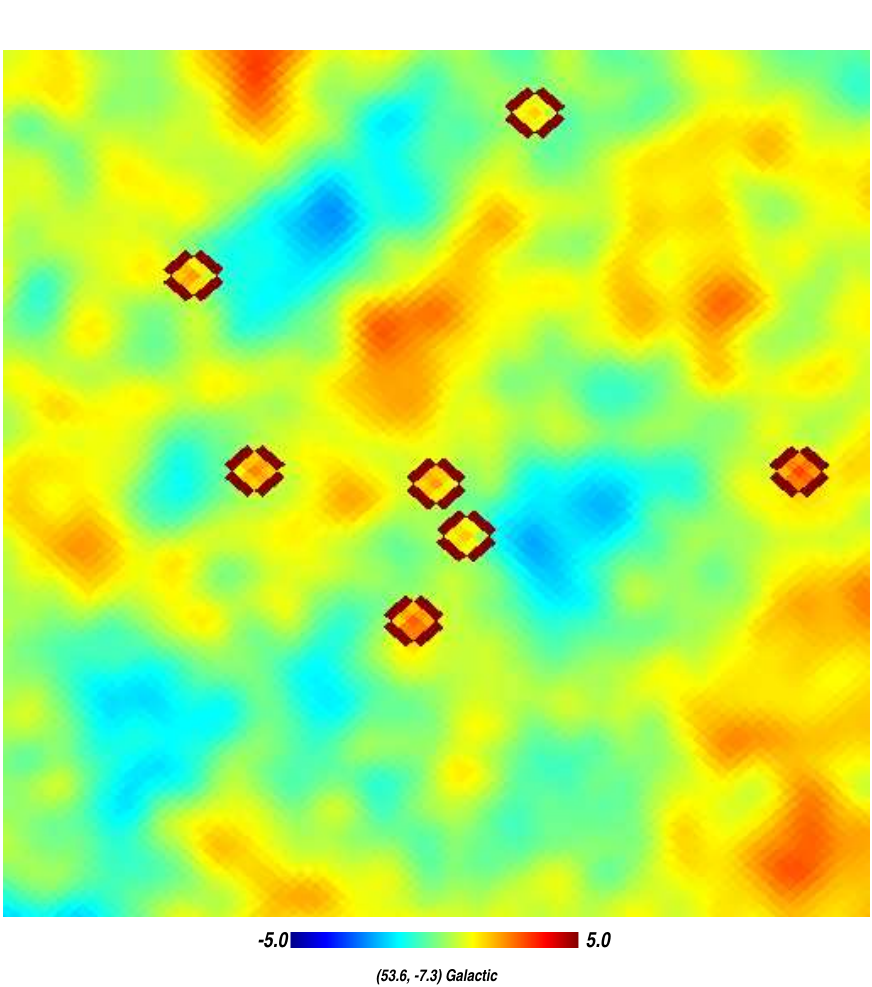}
  \includegraphics[width=0.49\textwidth,angle=0]{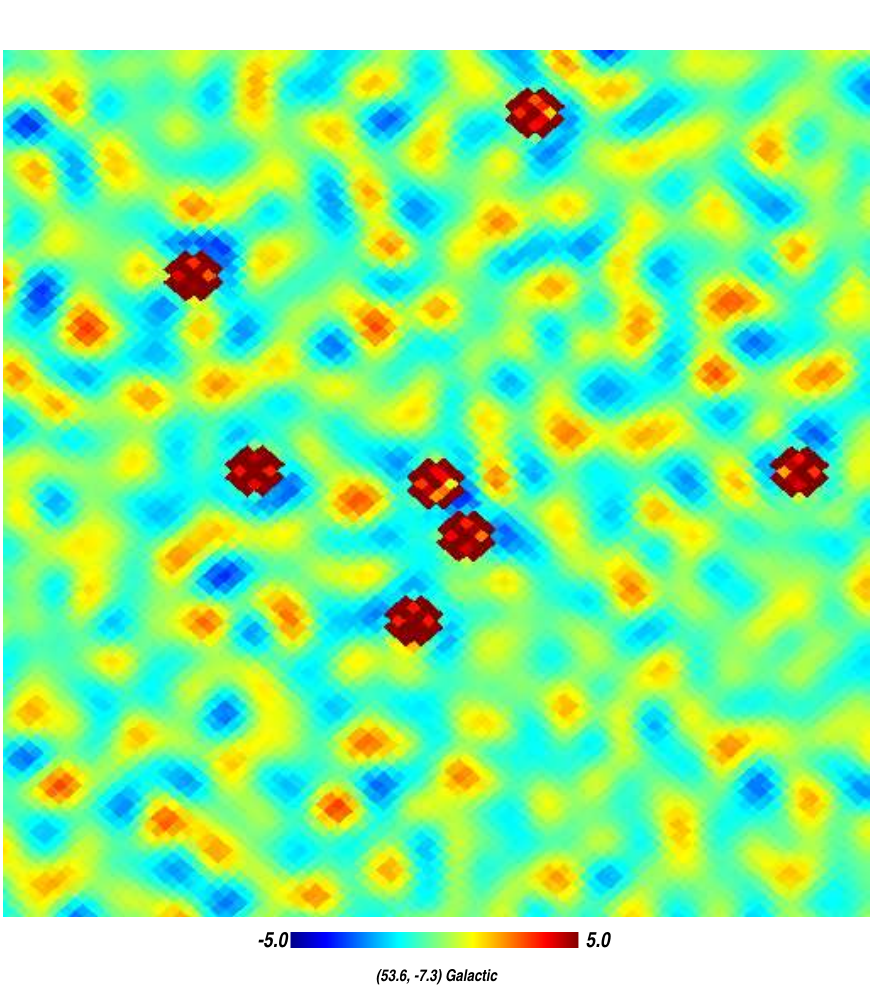}
\caption{{\rm \bf Signal plus noise maps:} Upper panel is an equal
  area stereographic projection (the so called Mollweide projection)
  of the signal plus noise simulation before needlet filtering  The color map is given in standardized RMS units.  Bottom left
  panel is a gnomonic projection of the unfiltered map around a point
  source with 5 degree diameter; bottom right panel is a similar
  gnomonic projection around the same point source but from a Mexican
  needlet filtered map.  The Mexican needlet parameters used are
  $j=38$, $B=1.2$ and $p=1$. The red marks have been added only to visualize the location of the point sources but are not part of the simulation.
  \label{fig:maps}}
\end{center}
\end{figure}

\subsection{Distribution of peak heights}
\begin{figure}[H]
\begin{center}
\includegraphics[width=0.49\textwidth,angle=0]{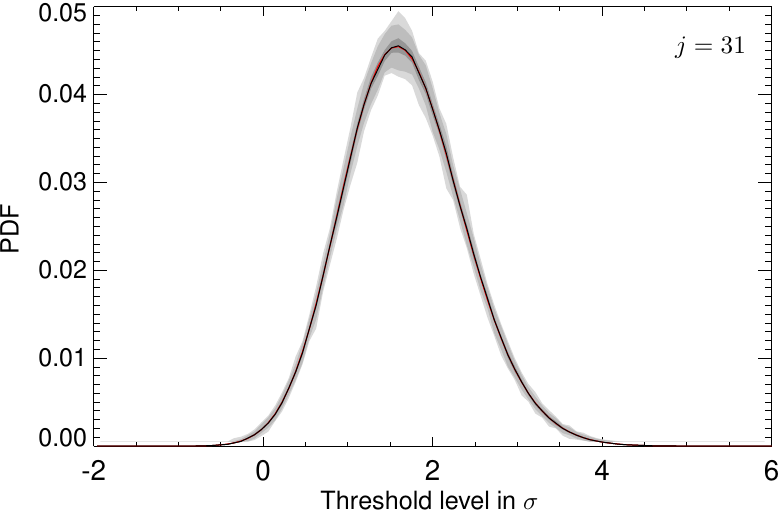}
\includegraphics[width=0.49\textwidth,angle=0]{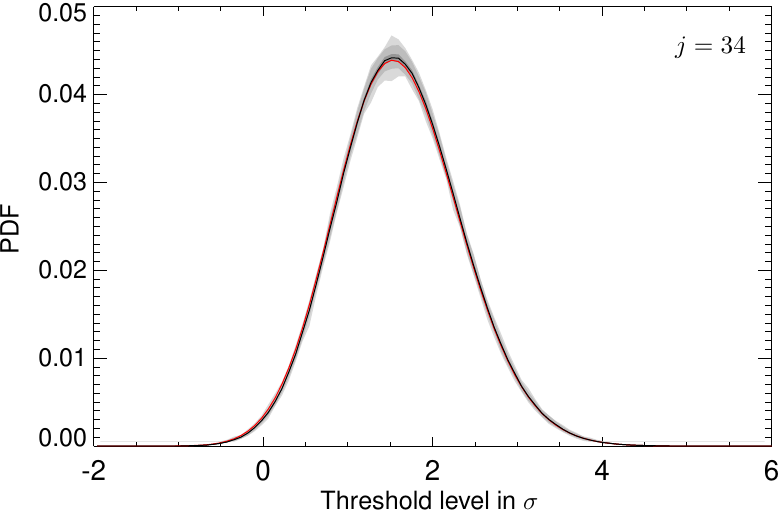}
\includegraphics[width=0.49\textwidth,angle=0]{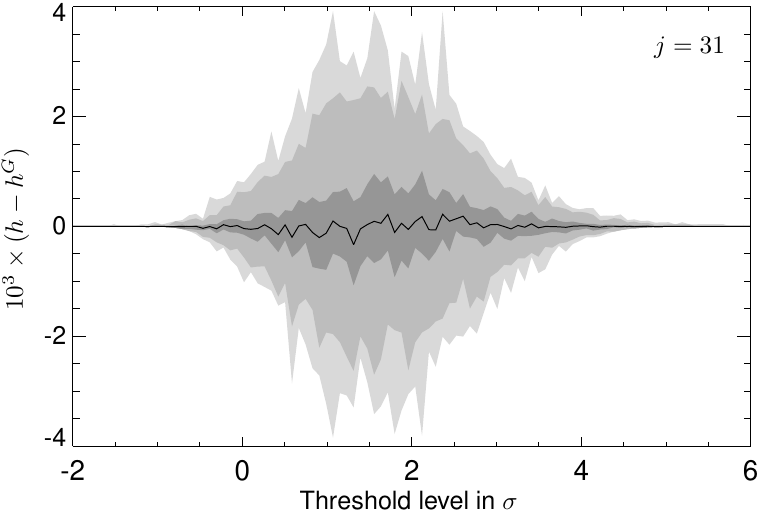}
\includegraphics[width=0.49\textwidth,angle=0]{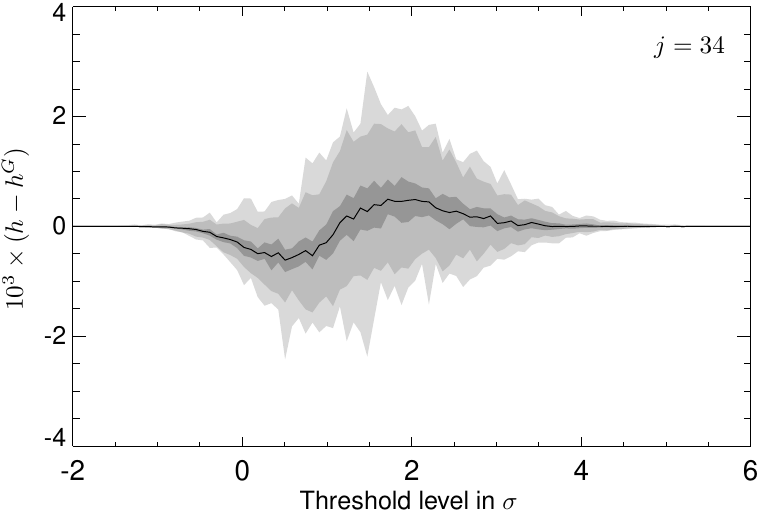}
\caption{{\rm \bf Peaks PDF:} probability density of height of local maxima.
In the upper panels, the red curves represent the analytical values, while
the black curves and the gray contours are the mean and the 68,95,99
\% percentiles from the simulations from 100 Monte Carlo simulations with no point sources. The lower panels show the difference between the
  analytical ($h^G$) and numerical ($h$) result. The mexican needlet
  parameters used are $p=1$, $B=1.2$ and $j=31,34$, which corresponds to
  central multipoles of $\ell = [284,492]$.  \label{fig:pdf1}}
\end{center}
\end{figure}

The theoretical distribution of local maxima (peaks) on a Mexican
needlet filtered Gaussian map is given by \refeq{eq:height-distr}. In
\reffig{fig:pdf1} we present the comparison of the theoretical
density, $h_j(x)$, with what we obtained empirically using 100
Gaussian map simulations with no point sources.  The upper panels from
left to right respectively present the normalized Gaussian peak PDFs
for needlet frequency $j=31,34$, with Mexican needlet parameters
$B=1.2$ and $p=1$. We chose these values as a natural compromise which
on one hand illustrates higher multipoles behaviour, on the other hand
still allows for extremely good numerical accuracy (better than 1\%
precision for finding peaks).


In the lower panel of \reffig{fig:pdf1} we show the relative
percentage difference between the analytical and simulation
results. It is easy to see from these figures that the theory fits the
numerical results remarkably well. Moreover, the dispersion around the
expected value of the PDF decreases as $j$ increases, consistently
with the ergodicity result of Theorem \ref{var_noise}.

\subsection{Application of the STEM algorithm}
The first step in the STEM algorithm, after needlet filtering, is to normalize the map using its standard deviation, as defined in \refeq{eq:standardize}, to obtain \refeq{eq:signal+noise:norm}.

To find local peaks on a map we compute the first and second
derivatives using \healpix's routine \emph{alm2map\_der}. The
pixels where the first derivative is close to zero (within a precision of $10^{-6}$) are classified as the local extrema.  We then partition these extrema
into maxima, minima and saddle using the eigenvalue decomposition of
the Hessian matrix - of course, maxima are those with all the eigenvalues
negative.


It is instructive to look at how the brightness of point sources
increase as we filter the signal plus noise map with Mexican needlets.
In \reffig{fig:hist}, we plot the PDF of point source amplitudes
before adding noise (grey curve), after adding noise but before
needlet filtering (thick black curve), and after filtering with
increasing $j$. For the high frequency Mexican needlet we considered,
$j=38$, filtering increases the brightness by a factor greater than
4. The negative values in the histogram are due to the added Gaussian
noise; we do not expect to detect such weak sources based on their
amplitude information only.

\begin{figure}[H]
\begin{center}
  \includegraphics[width=0.9\textwidth,angle=0]{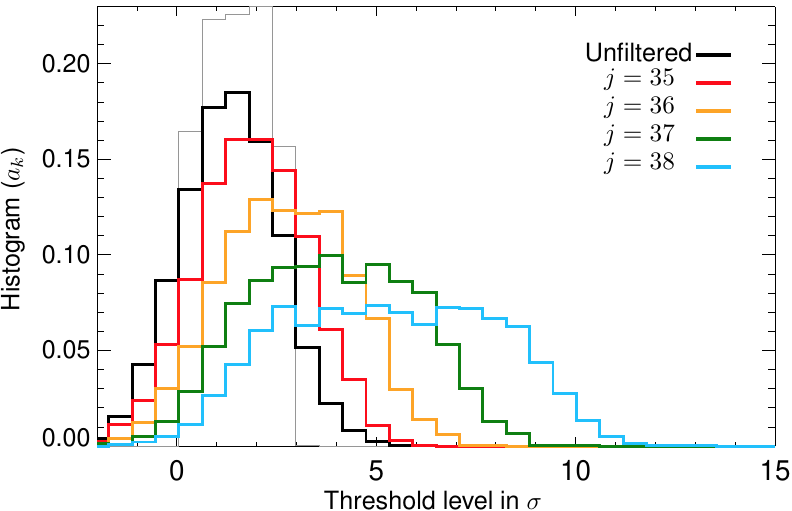}
\caption{{\rm \bf Needlet filtering increases signal-to-noise ratio:} Histogram
  of signal amplitudes at the location of the point sources, before adding noise (grey curve), after adding noise but before filtering (thick black curve),
  and after filtering with the Mexican needlet at different $j$. The Mexican needlet
  parameters used are $B=1.2$ and $p=1$. \label{fig:hist}}
\end{center}
\end{figure}

\begin{figure}[H]
\begin{center}
  \includegraphics[width=0.49\textwidth,angle=0]{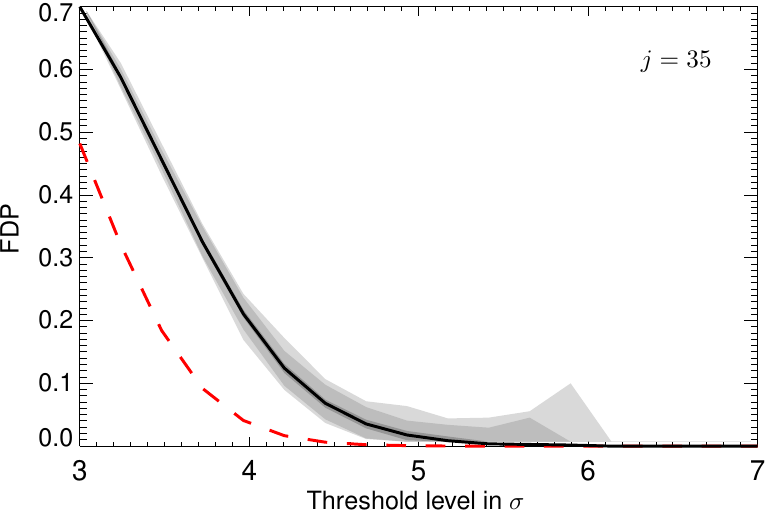}
  \includegraphics[width=0.49\textwidth,angle=0]{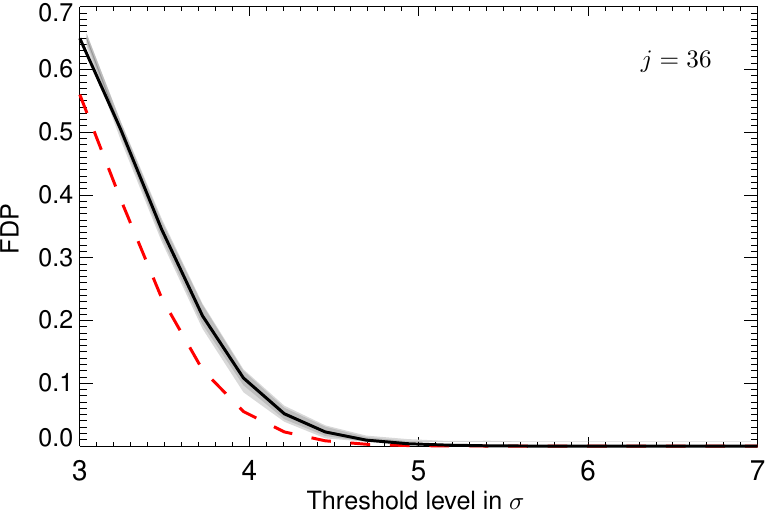}
  \includegraphics[width=0.49\textwidth,angle=0]{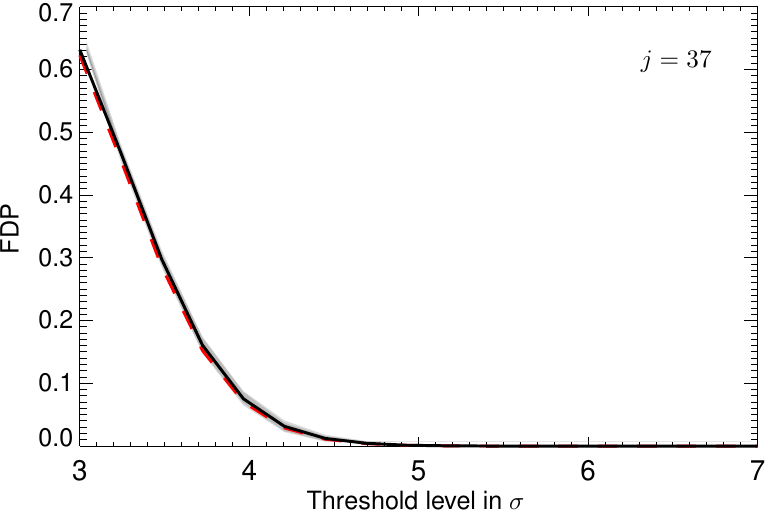}
  \includegraphics[width=0.49\textwidth,angle=0]{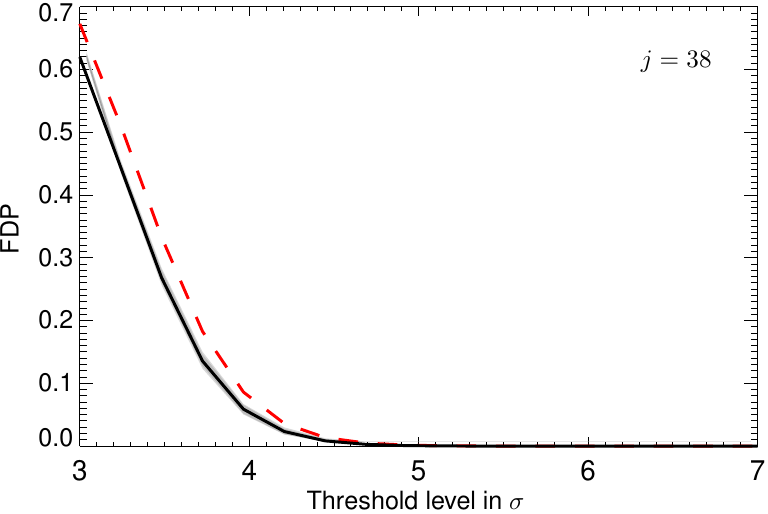}
\caption{{\rm \bf False Discovery Proportion:} FDP as a function of threshold (in units of standard deviation) for different needlet scales.  The dashed red
  curves are for the analytical upper bounds while the black curves are for the mean
  of the empirical FDPs from 100 Monte Carlo simulations. The gray shades are
  for percentiles $68, 95$ and $99\%$.  The Mexican needlet parameters
  used are $B=1.2$ and $p=1$. \label{fig:fdp1}}
\end{center}
\end{figure}

\subsection{False Discovery Proportion (FDP)}
In \refeq{eq:bound-FDR} of \refsec{S:proofFDR} we provided the
analytical result on the upper bound of the FDP as a function of the
power spectrum of the noise, the total number and the spatial profile
of the sources. Here we compare this result with what is obtained from
numerical simulations.

The empirical FDP is computed using the following steps: locate maxima
on needlet filtered signal-plus-noise Monte Carlo simulations using
our peak detection code; classify peaks as \emph{True discovery} if
the location of a maxima corresponds to a known (input) point source
within $\rho$ pixel radius or \emph{False discovery} if there are no
input sources within $\rho$ pixels radius of the peak ($\rho$
corresponding to the tolerance parameter); count the number of True
and False discoveries as a function of $\rho$ and the RMS of the noise.

\begin{figure}[H]
\begin{center}
  \includegraphics[width=0.49\textwidth,angle=0]{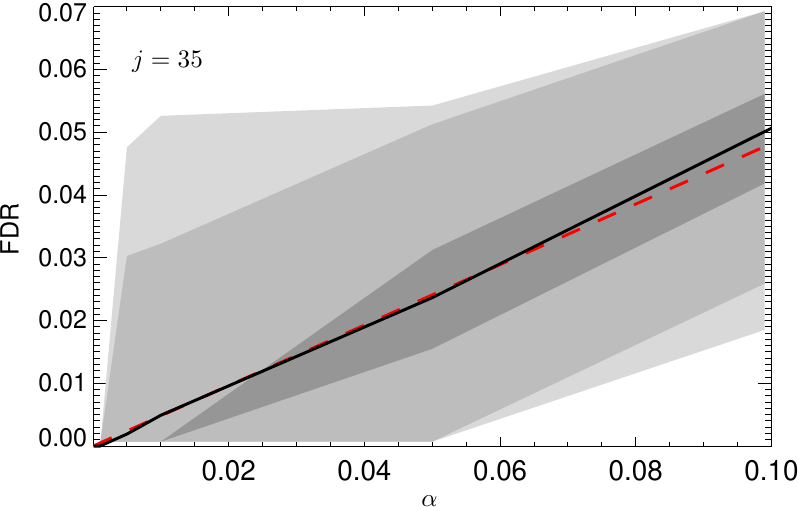}
  \includegraphics[width=0.49\textwidth,angle=0]{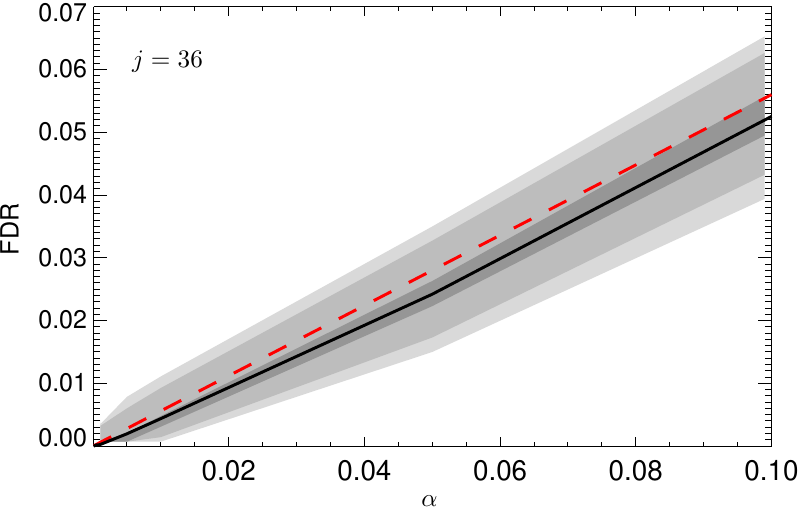} \\
  \includegraphics[width=0.49\textwidth,angle=0]{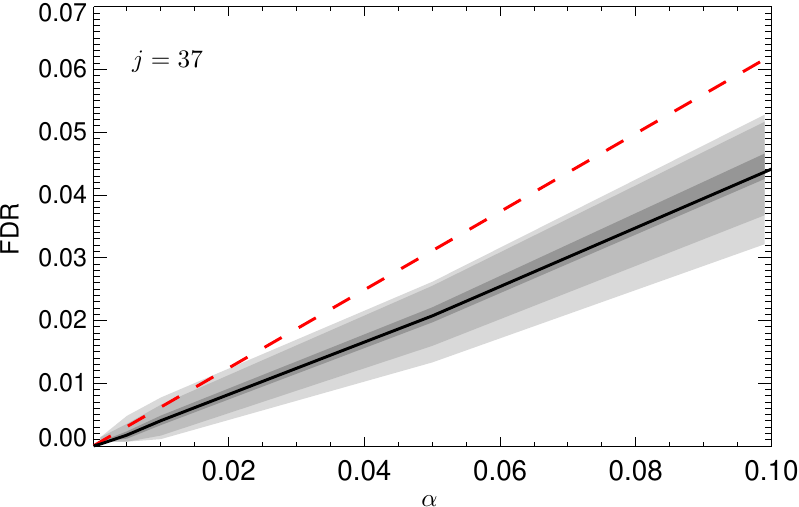}
  \includegraphics[width=0.49\textwidth,angle=0]{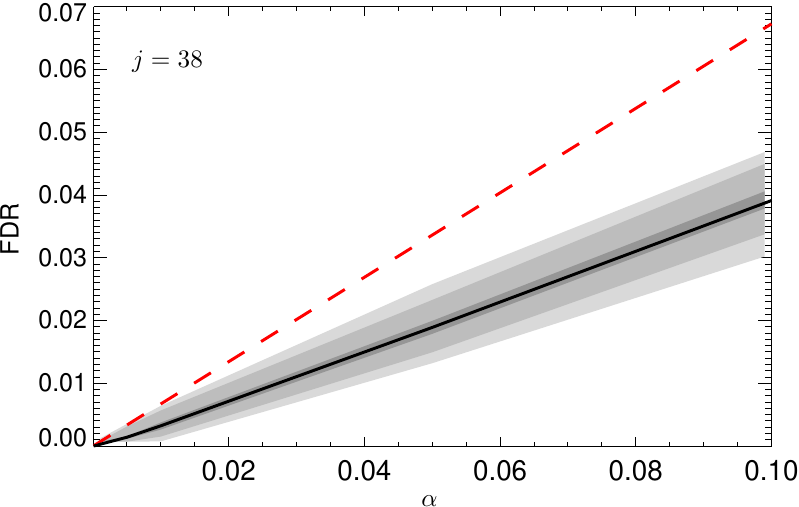}
\caption{{\rm \bf False Discovery Rate:} FDR as a function of error
  rate $\alpha$ for different needlet scales. The red curve is a plot
  of $\alpha*FDP(u=3)$, while the black curve is the mean from 100
  Monte Carlo simulations. The gray shades are for percentiles $68,
  95$ and $99\%$. The number of point sources is 5000. The Mexican
  needlet parameters used are $B=1.2$ and $p=1$. \label{fig:fdr1}}
\end{center}
\end{figure}

The empirical FDP as a function of $u$, which is in units of the RMS
of the noise, and the source detection tolerance parameter $\rho$ is
computed, according to \eqref{eq:FDP}, as
\begin{equation}\label{eq:empfdp}
\widehat{\mathrm{FDP}}_{\rho_j}(u) = \frac{\text{\# of False discoveries above\,}u}{\text{total \# of peaks above\,} u}
\end{equation}

In \reffig{fig:fdp1} we illustrate the comparison of the FDP for
thresholds in the filtered map above $3\sigma$ and $\rho=3$ for
different values of $j$. We found that setting $2\le \rho \le 8$ does
not alter significantly our results (note that the smallest practical
radius is $\rho = 2$). The red curve in these plots corresponds to the
first term on the right hand side of \refeq{eq:bound-FDR}, while the
black curve is from the mean of the simulations. The contours from
dark to light gray corresponds to the $68, 95$ and $99 \%$ confidence
intervals. We note that, as expected, the analytic results for the
upper bound become larger than the numerical simulations as $j$
increases.

\begin{figure}[H]
\begin{center}
  \includegraphics[width=0.49\textwidth,height=0.32\textwidth,angle=0]{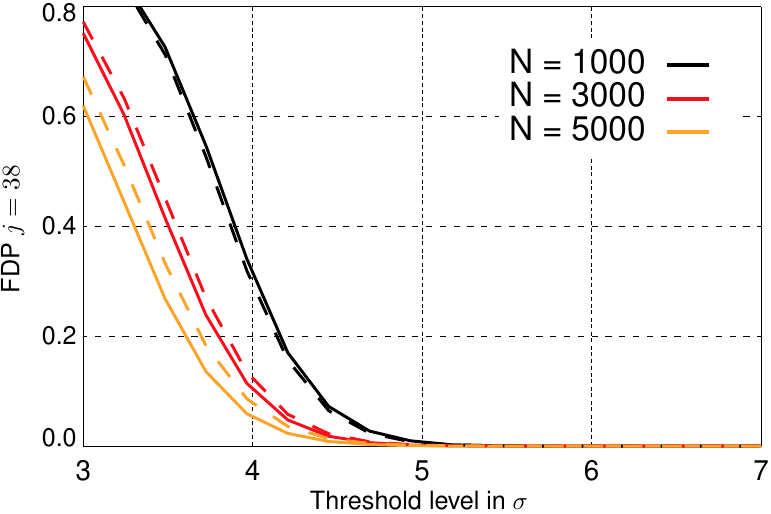}
  \includegraphics[width=0.49\textwidth,height=0.32\textwidth,angle=0]{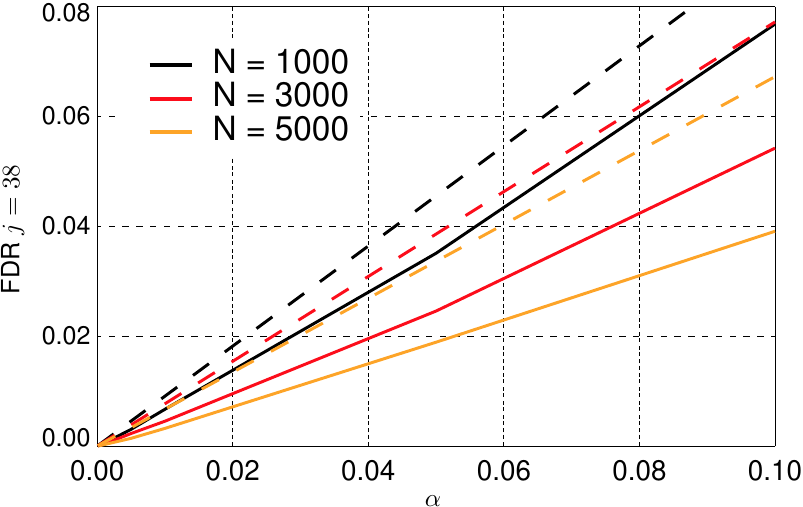}
  \caption{{\rm \bf Theoretical vs numerical results for FDP and FDR:} FDP as a function of
    thresholds and FDR as a function of global p-value, $\alpha$, for needlet scale $j=38$.
  The dashed curves are the analytical upper bounds while the solid curves are the
  corresponding empirical results from 100 Monte Carlo
  simulations. The Mexican needlet parameters used are $B=1.2$ and
  $p=1$ \label{fig:fdp_fdr1}}
\end{center}
\end{figure}

\subsection{False Discovery Rate (FDR)}

We now proceed in validating the analytical formalism established in
\refsec{S:proofFDR} to control the false discovery rate (FDR). This is
done by comparing the analytical upper bound of the FDR, which is
given by \refeq{eq:bound-FDR}, with the empirical result from
simulations. In \reffig{fig:fdr1}, it is shown that for a given
error rate, the empirical FDR is always below the upper limit set by
the theory.

In \reffig{fig:fdp_fdr1} we present the mean FDP and FDR curves
together with the corresponding theoretical results for different
number of input sources. Again, the FDP and FDR are bounded above by
the theoretical bounds.

\subsection{Detection power}

To quantify how many of the input point sources we discovered in our
analysis, in \reffig{fig:td} we show the number of peaks that matches
the true sources i.e., the numerator of \eqref{eq:power}, which
measures the statistical power of the algorithm.  These results show
that the power of the STEM algorithm is almost 100\% in detecting
bright sources - indeed, we have detected all input sources whose
brightness was above $1\sigma$ in the unfiltered simulated maps.

Overall, we believe that the results in this section provide a strong
numerical support for the asymptotic findings that we described
earlier in this paper.

\begin{figure}[H]
\begin{center}
  \includegraphics[width=0.95\textwidth,angle=0]{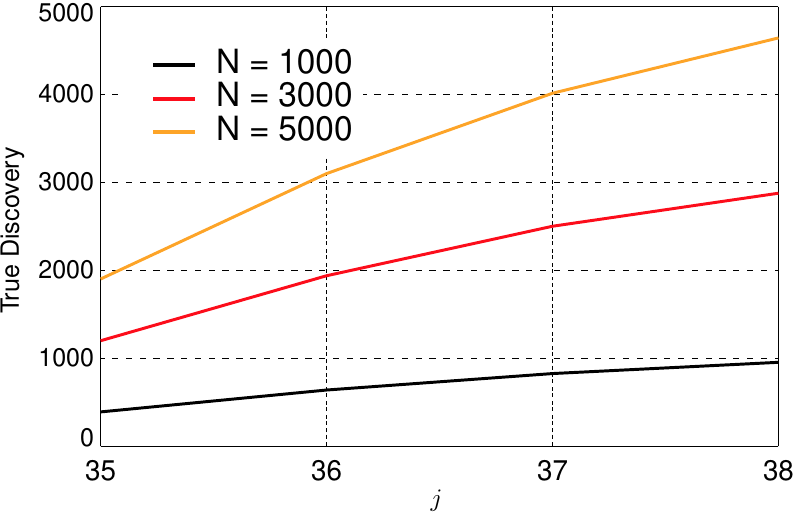}
  \caption{{\rm \bf Detection power:} number of true
    discoveries for threshold $u>3$ as a function of needlet scales,
    $j$.  The three curves are the mean of 100 Monte Carlo simulations
    for the corresponding cases. The legend shows the number of input
    point sources in simulations.  The Mexican needlet parameters used
    are $B=1.2$ and $p=1$. \label{fig:td}}
\end{center}
\end{figure}

\appendix

\section{Proof of Theorem \protect\ref{var_noise}}

\subsection{Voronoi cells}

We introduce the following notation for the spherical caps in $\mathbb{S}^2$:
\begin{align*}
\mathcal{B}(a,\varepsilon)=\{x \subseteq \mathbb{S}^2: d(a,x) \le
\varepsilon\}.
\end{align*}
For any $\varepsilon>0$, we say that $\Xi_\varepsilon=\{\xi_{1,\varepsilon},
\dots, \xi_{N,\varepsilon}\}$ is a \textit{maximal $\varepsilon$-net}, if $%
\xi_{1,\varepsilon}, \dots, \xi_{N,\varepsilon}$ are in $\mathbb{S}^2$, $%
\forall i \ne j$ we have $d(\xi_{i,\varepsilon},
\xi_{j,\varepsilon})>\varepsilon$ and $\forall x \in \mathbb{S}^2,\;\;
d(x,\Xi_\varepsilon) \le \varepsilon.$ Heuristically, an $\varepsilon$-net
is a grid of point at a distance at least $\varepsilon$ from each other, and
such that any extra point should be within a distance $\varepsilon$ from a
point in the grid, see \cite[Lemma 5]{BKMP}. The number $N$ of points in a $%
\varepsilon$-net on the sphere can be bounded from above and from below,
indeed we have the following:
\begin{equation*}
\frac{4}{\varepsilon^2} \le N \le \frac{4}{\varepsilon^2} \pi^2.
\end{equation*}
Given an $\varepsilon$-net it is natural to partition the sphere into
disjoint sets, each of them associated with a single point in the net. This
task is accomplished by the well-known Voronoi cells construction.

\begin{definition}
Let $\Xi_\varepsilon$ be a maximal $\varepsilon$-net. For all $%
\xi_{i,\varepsilon} \in \Xi_\varepsilon$, the associated family of Voronoi
cells is defined by
\begin{equation*}
\mathcal{V}(\xi_{i,\varepsilon},\varepsilon)=\{x \in \mathbb{S}^2: \forall
j \ne i, \; d(x, \xi_{i,\varepsilon}) \le d(x,\xi_{j,\varepsilon})\}.
\end{equation*}
\end{definition}

\noindent We recall that $\mathcal{B}(\xi_{i,\varepsilon},\varepsilon/2)
\subseteq \mathcal{V}(\xi_{i,\varepsilon},\varepsilon) \subseteq \mathcal{B}%
(\xi_{i,\varepsilon},\varepsilon)$, hence $\text{Area}(\mathcal{V}%
(\xi_{i,\varepsilon},\varepsilon)) \approx \varepsilon^2$. Let
\begin{equation*}
\mathcal{N}^c(\tilde{\beta}_j; \mathcal{V}(\xi_{i,\varepsilon},\varepsilon),I)=\#\{x
\in \mathcal{V}(\xi_{i,\varepsilon},\varepsilon): \tilde{\beta}_j(x) \in I, \nabla
\tilde{\beta}_j(x)=0\}.
\end{equation*}
Note that, almost surely, the sum of the critical points over the
Voronoi cells equals the total number of critical points:
\begin{align*}
\mathcal{N}^c_I(\tilde{\beta}_j)= \sum_{\xi_{i,\varepsilon} \in \Xi_\varepsilon }
\mathcal{N}^c(\tilde{\beta}_j; \mathcal{V}(\xi_{i,\varepsilon},\varepsilon),I).
\end{align*}

Our proof uses similar ideas to those exploited in \cite{cmw2014} for the analysis of critical points of spherical random eigenfunctions. In particular, we split the variance into two terms, one related to the correlation between Voronoi cells which are further apart than an (asymptotically vanishing) threshold (the so-called \textquotedblleft long-range component"), the other related to Voronoi cells whose distance is smaller than the threshold.

More precisely, we have that
\begin{align}
\text{Var} \left( \mathcal{N}^{c}_I(\tilde{\beta}_j)\right)
&=\sum_{\xi_{i,\varepsilon}, \xi_{k,\varepsilon} \in \Xi_\varepsilon } \text{%
Cov} \left( \mathcal{N}^{c}(\tilde{\beta}_j; \mathcal{V}(\xi _{i,\varepsilon }),I),%
\mathcal{N}^{c}(\tilde{\beta}_j; \mathcal{V}(\xi _{k,\varepsilon }),I)\right)  \notag
\\
& = \sum_{d(\mathcal{V}(\xi _{i,\varepsilon }), \mathcal{V}(\xi
_{k,\varepsilon }))> C/B^j } \text{Cov} \left( \mathcal{N}^{c}(\tilde{\beta}_j;
\mathcal{V}(\xi _{i,\varepsilon }),I), \mathcal{N}^{c}(\tilde{\beta}_j; \mathcal{V}%
(\xi _{k,\varepsilon }),I)\right)  \label{mir1} \\
&\;\;+ \sum_{d( \mathcal{V}(\xi _{i,\varepsilon }), \mathcal{V}(\xi
_{k,\varepsilon })) \le C/B^j } \text{Cov} \left( \mathcal{N}^{c}(\tilde{\beta}_j;%
\mathcal{V}(\xi _{i,\varepsilon }),I), \mathcal{N}^{c}(\tilde{\beta}_j;\mathcal{V}%
(\xi _{k,\varepsilon }),I)\right).  \label{mir2}
\end{align}

In Section \ref{long} we prove that the asymptotic behaviour of the
long-range component \eqref{mir1} is
\begin{equation}
\sum_{d(\mathcal{V}(\xi _{i,\varepsilon }), \mathcal{V}(\xi _{k,\varepsilon
})) > C/B^j } \text{Cov} \left( \mathcal{N}^{c}(\tilde{\beta}_j; \mathcal{V}(\xi
_{i,\varepsilon }),I), \mathcal{N}^{c}(\tilde{\beta}_j; \mathcal{V}(\xi
_{k,\varepsilon }),I)\right) \leq c_1(I) j^2 B^{2j}+o( j^2 B^{2j})
\label{AA}
\end{equation}
while in Section \ref{short} we prove that for \eqref{mir2} we have
\begin{equation}
\sum_{d( \mathcal{V}(\xi _{i,\varepsilon }), \mathcal{V}(\xi _{k,\varepsilon
})) \le C/B^j } \text{Cov} \left( \mathcal{N}^{c}(\tilde{\beta}_j;\mathcal{V}(\xi
_{i,\varepsilon }),I), \mathcal{N}^{c}(\tilde{\beta}_j;\mathcal{V}(\xi
_{k,\varepsilon }),I)\right)\leq c_2(I) B^{2j}+o(B^{2j}),
\label{BB}
\end{equation}
where $c_1(I),c_2(I)$ are uniformly bounded for every $I \subset \mathbb{R}$.



As in \cite{cmw2014}, let us introduce also the two-point correlation function $K_{2,j}$, which is given by
\begin{align}  \label{two-point}
K_{2,j}(x,y; t_1, t_2)=&\mathbb{E}\left[ \left\vert \nabla ^{2}
\tilde{\beta}_j(x)\right\vert \cdot \left\vert \nabla ^{2} \tilde{\beta}_j (y)\right\vert %
\Big| \nabla \tilde{\beta}_j(x)=\nabla \tilde{\beta}_j(y)=0, \tilde{\beta}_j(x)=t_1, \tilde{\beta}_j(y)=t_2 %
\right]  \notag \\
& \times \varphi _{x,y}(t_1 , t_2, 0,0,0,0)
\end{align}
where $t_1,t_2 \in \mathbb{R}$ and $\varphi _{x,y}(t_1 , t_2,0,0,0,0)$ denotes the density of the
6-dimensional vector
\begin{equation*}
\left( \tilde{\beta}_j(x), \tilde{\beta}_j(y),\nabla \tilde{\beta}_j(x),\nabla \tilde{\beta}_j(y)\right)
\end{equation*}
in $\tilde{\beta}_j(x)=t_1, \tilde{\beta}_j(y)=t_2,\nabla \tilde{\beta}_j(x)=\nabla \tilde{\beta}_j(y)=%
\mathbf{0}$. Note that, by isotropy, the function $K_{2,j}$ depends on
the points $x$, $y$ only through their geodesic distance $\phi=d(x,y)$; with
some abuse of notation, we shall write
\begin{equation*}
K_{2,j}(\phi; t_1 , t_2)=K_{2,j}(x,y;t_1,t_2).
\end{equation*}

We are now in the position to investigate the asymptotic behaviour of the long- and short-range components, respectively.

\subsection{Proof of the long-range asymptotic bound \eqref{AA}}

\label{long}

By the Kac-Rice expectation metatheorem (see i.e., \cite{RFG}, Chapter 11), we have
\begin{align*}
&\sum_{d( \mathcal{V} (\xi _{i,\varepsilon }), \mathcal{V}(\xi
_{k,\varepsilon }) ) > C/B^j} \text{Cov} \left( \mathcal{N}^{c}(\tilde{\beta}_j;
\mathcal{V}(\xi _{i,\varepsilon }),I),\mathcal{N}^{c}(\tilde{\beta}_j; \mathcal{V}%
(\xi _{k,\varepsilon }),I) \right) \\
&\hspace{1,5cm}= \int_{\mathcal{W}}\int_{I\times I} K_{2,j}(x,y;t_1,t_2) d
t_1 d t_2 dxdy - \mathbb{E}[\mathcal{N}^{c}(\tilde{\beta}_j; \mathcal{V}(\xi
_{i,\varepsilon }),I)] \mathbb{E}[\mathcal{N}^{c}(\tilde{\beta}_j; \mathcal{V}(\xi
_{k,\varepsilon }),I)]
\end{align*}
where
\begin{equation*}
\mathcal{W}=\bigcup _{d(\mathcal{V}(\xi _{i,\varepsilon }), \mathcal{V}(\xi
_{k,\varepsilon }))> C/B^j }\mathcal{V}(\xi _{i,\varepsilon })\times
\mathcal{V}(\xi _{k,\varepsilon }),
\end{equation*}
is the union of Voronoi cells which are further away than $C/B^{j}$.
Our next result is a convenient expression for the expectation; here and in the sequel, we use the simplified notation
$${\cal B}_{2n}={\cal B}_{2n,j}={\cal B}_{2n,p,j},$$ where ${\cal B}_{2n,p,j}$ was introduced in \eqref{B_not}.

\begin{lemma}
\label{exp}
\begin{align*}
\mathbb{E}[\mathcal{N}_{I}^{c}(\tilde{\beta}_j(x))]=\frac 1 2 \frac{\mathcal{B}_{4,j}%
}{ \mathcal{B}_{2,j} } \int_{I} p(t) d t,
\end{align*}
where the function $p$ is defined in \eqref{f_p}.
\end{lemma}

\begin{proof}
\medskip
Let $\mathcal{N}_{I}^{c}(\tilde{\beta}_j)$ be the number of critical
points with value in $I$
\begin{align*}
\mathcal{N}_{I}^{c}(\tilde{\beta}_j)&=\#\{x\in {S^2}: \tilde{\beta}_j(x)\in I, \nabla
\tilde{\beta}_j(x)=0\}.
\end{align*}
From isotropy and Kac-Rice metatheorem we immediately get that
\begin{align*}
\mathbb{E}[\mathcal{N}_{I}^{c}(\tilde{\beta}_j(x))]&=4 \pi \int_{\mathbb{R}^4}
|\zeta_{1} \zeta_{3} -\zeta_{2}^2| \mathbbm{1}_{\{t \in I \}} D_{j} (t,0
,0,\zeta_{1},\zeta_{2}, \zeta_{3}) d t \; d \zeta_{1}\; d \zeta_{2} \; d
\zeta_{3}
\end{align*}
where $D_{j} (t,0 ,0,\zeta_{1},\zeta_{2}, \zeta_{3})$ denotes the joint
density of $(\tilde{\beta}_j(x), \nabla \tilde{\beta}_j(x), \nabla^2 \tilde{\beta}_j(x))$ in $%
\tilde{\beta}_j(x)=t$, $\nabla \tilde{\beta}_j(x)=\mathbf{0}$, $\nabla^2
\tilde{\beta}_j(x)=(\zeta_{1},\zeta_{2}, \zeta_{3})={\zeta}$. Since, at each fixed $x
$, first and second derivatives are uncorrelated we have
\begin{equation*}
D_{j} (t,0 ,0,\zeta_{1},\zeta_{2}, \zeta_{3})=D_{j,1} (0 ,0) D_{j,2}(t)
D_{j,3}(\zeta_{1},\zeta_{2}, \zeta_{3}| \tilde{\beta}_j(x)=t),
\end{equation*}
where $D_{j,1}$, $D_{j,2}$ and $D_{j,3}$ are the marginal densities of $%
\nabla \tilde{\beta}_j(x)$, $\tilde{\beta}_j(x)$ and $( \nabla^2 \tilde{\beta}_j(x) | \tilde{\beta}_j(x)=t)$
respectively. In view of the results in Section \ref{cov_mat_section}, we
immediately have that
\begin{equation*}
D_{j,1} (0 ,0)=\frac{1}{2 \pi} \frac{1}{\sqrt{ \frac{1}{4} \mathcal{B}%
^2_{2,j}}} =\frac{1}{2 \pi} \frac{1}{\frac{ \mathcal{B}_{2,j}}{2}}, \hspace{%
1cm} D_{j,2}(t)=\frac{1}{\sqrt{2 \pi}} e^{-\frac{t^2}{2}},
\end{equation*}
and
\begin{align*}
D_{j,3}(\zeta_{1},\zeta_{2}, \zeta_{3}| \tilde{\beta}_j(x)=t)=\frac{1}{(2 \pi)^{3/2}
\sqrt{ \text{det} (\omega_j) }} \exp\{-\frac 1 2 ({\zeta}-\mu_j(t))%
\omega^{-1}_j ({\zeta} - \mu_j(t))^t\}
\end{align*}
where
\begin{align*}
\mu_j(t)=\left(
\begin{array}{c}
-\frac{t}{2} \mathcal{B}_{2,j} \\
0 \\
-\frac{t}{2} \mathcal{B}_{2,j}%
\end{array}
\right), \hspace{0.5cm} \omega_j= \left(
\begin{array}{ccc}
\frac{3}{8} \mathcal{B}_{4,j} -\frac{1}{4} \mathcal{B}_{2,j} -\frac 1 4
\mathcal{B}^2_{2,j} & 0 & \frac{1}{8} \mathcal{B}_{4,j} +\frac{1}{4}
\mathcal{B}_{2,j} -\frac 1 4 \mathcal{B}^2_{2,j} \\
0 & \frac{1}{8} \mathcal{B}_{4,j} -\frac{1}{4} \mathcal{B}_{2,j} & 0 \\
\frac{1}{8} \mathcal{B}_{4,j} +\frac{1}{4} \mathcal{B}_{2,j} -\frac 1 4
\mathcal{B}^2_{2,j} & 0 & \frac{3}{8} \mathcal{B}_{4,j} -\frac{1}{4}
\mathcal{B}_{2,j} -\frac 1 4 \mathcal{B}^2_{2,j}%
\end{array}%
\right).
\end{align*}
With the scaling $\frac{\sqrt 8}{ \sqrt{\mathcal{B}_{4,j}} }( \nabla^2
\tilde{\beta}_j| \tilde{\beta}_j)\sim N(\tilde{\mu}_j(t), \tilde{\omega}_j)$, where $\tilde{\mu}_j(t)=\frac{\sqrt 8}{ \sqrt{\mathcal{B}_{4,j}} }  \mu_j(t)$ and $\tilde{\omega}_j= \frac{8}{ {\mathcal{B}_{4,j}} } \omega_j$,
we obtain
\begin{align*}
&D_{j} (t,0 ,0,\zeta_{1},\zeta_{2}, \zeta_{3}) d t \; d \zeta_{1} \; d
\zeta_{2} \; d \zeta_{3} \\
&= \frac{1}{2 \pi} \frac{1}{\frac{ \mathcal{B}_{2}}{2}} \frac{1}{\sqrt{2
\pi}} e^{-\frac{t^2}{2}} \frac{1}{(2 \pi)^{3/2} \sqrt{\text{det} (\tilde{%
\omega}_j)}} \exp\{-\frac 1 2 (\tilde{\zeta}-\tilde{\mu}_j(t)) \tilde{\omega}%
_j^{-1} (\tilde{\zeta}-\tilde{\mu}_j(t))^t\} d t d \tilde{\zeta}_{1} d
\tilde{\zeta}_{2} d \tilde{\zeta}_{3}
\end{align*}
and then
\begin{align*}
\int_{\mathbb{R}^4} |\zeta_{1} \zeta_{3} -\zeta_{2}^2| \mathbbm{1}_{\{t \in
I \}} D_{j} (t,0 ,0,\zeta_{1},\zeta_{2}, \zeta_{3}) d t d \zeta_{1} d
\zeta_{2} d \zeta_{3}=\frac{\mathcal{B}_{4,j}}{ 8 \frac{ \mathcal{B}_{j;2}}{2%
}}\frac{1}{2 \pi} \int_{I} p(t) d t,
\end{align*}
with
\begin{align}  \label{f_p}
p(t)=\frac{1}{\sqrt{2 \pi}} \frac{1}{(2 \pi)^{3/2}} \frac{1}{ \sqrt{\text{det%
}( \tilde{\omega}_j )}} \int_{\mathbb{R}^3} |\tilde{\zeta}_{1} \tilde{\zeta}%
_{3} -\tilde{\zeta}_{2}^2| e^{-\frac{t^2}{2}} \exp\{-\frac 1 2 (\tilde{\zeta}%
-\tilde{\mu}_j(t)) \tilde{\omega}_j^{-1} (\tilde{\zeta}-\tilde{\mu}%
_j(t))^t\} d \tilde{\zeta}_{1} d \tilde{\zeta}_{2} d \tilde{\zeta}_{3}.
\end{align}
\end{proof}
\noindent In order to bound the variance we need to prove the following:
\begin{align}  \label{finalr}
&16 \pi^2 \iint_{I \times I} \int_{C/B^j}^{\frac \pi 2} K_{2,j}(\phi; t_1,
t_2) \sin \phi d \phi d t_1 d t_2- \frac{1}{4} \frac{\mathcal{B}^2_{4,j}}{%
\mathcal{B}^2_{2,j}} \iint_{I \times I} p(t_1) p(t_2) d t_1 d t_2 \nonumber \\
&\;\;\leq c_1(I) j^2B^{2j}+o(j^2B^{2j})
\end{align}
as $j \to \infty$. The idea of the proof is to give a Taylor expansion of the difference between the two integrands in \eqref{finalr}, see \cite{cmw2014} for a related argument. More precisely, we start by giving the explicit expression of $K_{2,j}$
in terms of a Gaussian integral; indeed we write, for $\phi \ge C/B^j$,
\begin{align*}
K_{2,j} (\phi; t_1, t_2)&= \iint_{\mathbb{R}^3 \times \mathbb{R}^3}
|\zeta_{x,1} \zeta_{x,3} -\zeta_{x,2}^2| |\zeta_{y,1} \zeta_{y,3}
-\zeta_{y,2}^2| \\
&\;\; \times D_{j,x,y} (t_1, t_2 ,0 ,0,0,0,\zeta_{x,1},\zeta_{x,2},
\zeta_{x,3},\zeta_{y,1},\zeta_{y,2}, \zeta_{y,3}) d \zeta_{x,1} d
\zeta_{x,2} d \zeta_{x,3} d \zeta_{y,1} d \zeta_{y,2} d \zeta_{y,3}
\end{align*}
where $D_{j,x,y}$ denotes the joint density of $(\tilde{\beta}_j(x),\tilde{\beta}_j(y),
\nabla \tilde{\beta}_j(x),\nabla \tilde{\beta}_j(y), \nabla^2 \tilde{\beta}_j(x),\nabla^2 \tilde{\beta}_j(y))
$. We have also
\begin{align*}
D_{j,x,y} (t_1, t_2 ,\mathbf{0},\mathbf{0},\zeta_{x},\zeta_{y})& =
D_{j,x,y,1} (\mathbf{0},\mathbf{0}) D_{j,x,y,2} (t_1, t_2 | \nabla
\tilde{\beta}_j(x)=\nabla \tilde{\beta}_j(y)=\mathbf{0}) \\
&\;\; \times D_{j,x,y,3} (\zeta_{x}, \zeta_{y} |\tilde{\beta}_j(x)=t_1,
\tilde{\beta}_j(y)=t_2, \nabla \tilde{\beta}_j(x)=\nabla \tilde{\beta}_j(y)=\mathbf{0})
\end{align*}
where $D_{j,x,y,1}$ is the density of $(\nabla \tilde{\beta}_j(x),\nabla \tilde{\beta}_j(y))$
and $D_{j,x,y,2}$, $D_{j,x,y,3}$ are the conditional densities of $%
(\tilde{\beta}_j(x),\tilde{\beta}_j(y)| \nabla \tilde{\beta}_j(x)=\nabla \tilde{\beta}_j(y)=\mathbf{0})$ and
$(\nabla^2 \tilde{\beta}_j(x),\nabla^2 \tilde{\beta}_j(y)| \tilde{\beta}_j(x)=t_1, \tilde{\beta}_j(y)=t_2,
\nabla \tilde{\beta}_j(x)=\nabla \tilde{\beta}_j(y)=\mathbf{0} )$ respectively.

In order to investigate the asymptotic behaviour of these densities, we need first to write the block components of the full covariance matrix of the field and its (first and second order) derivatives. In particular, we introduce the following notation for the matrix $\Sigma_j (\phi(x,y))=\Sigma_j (\phi)$:
\begin{equation*}
\Sigma_j(\phi)=\left(
\begin{array}{ccc}
R_j(\phi) & E_j(\phi) & D_j(\phi) \\
E^T_j(\phi) & A_j(\phi) & B_j(\phi) \\
D^{T}_j(\phi) & B^{T}_j(\phi) & C_j(\phi)%
\end{array}
\right).
\end{equation*}

\noindent Here, $R_j(\phi(x,y))$ is the $2 \times 2$ covariance matrix of $(\tilde{\beta}_j(x),\tilde{\beta}_j(y))$;
$E_j(\phi)$ is the $2 \times 4$ covariance matrix between $(\tilde{\beta}_{j }(x),\tilde{\beta}_{j }(y))$ and
$(\nabla \tilde{\beta}_j(x), \nabla \tilde{\beta}_j(y))$;
$D_j(\phi)$ is the $2\times 6$  covariance matrix between $(\tilde{\beta}_{j }(x),\tilde{\beta}_{j }(y))$ and
$(\nabla^2 \tilde{\beta}_j(x), \nabla^2 \tilde{\beta}_j(y))$; $A_j(\phi)$ is the $4 \times 4$ covariance matrix of
$(\nabla \tilde{\beta}_j(x), \nabla \tilde{\beta}_j(y))$; $B_j(\phi)$ is the $4 \times 6$ covariance matrix between $(\nabla \tilde{\beta}_{j }(x),\nabla \tilde{\beta}_{j }(y))$ and $(\nabla^2 \tilde{\beta}_j(x), \nabla^2 \tilde{\beta}_j(y))$;
and finally for $C_j(\phi)$ is the $6\times 6$ covariance matrix of the vector $(\nabla^2 \tilde{\beta}_j(x), \nabla^2 \tilde{\beta}_j(y))$ (see Appendix \ref{cov_mat_section}).

With this notation in mind, can easily give the value of the bivariate density for levels, evaluated at the origin; we have
\begin{equation*}
D_{j,x,y,1} (\mathbf{0},\mathbf{0})
= \frac{1}{\pi^{2} \mathcal{B}_{2}^2} \frac{1}{\sqrt{ \big( 1- 4 \frac{%
\alpha_1^2}{\mathcal{B}_{2}^2} \big) \big(1- 4 \frac{\alpha_2^2}{\mathcal{B}_{2}^2%
}\big) }},
\end{equation*}
where $\alpha_1=\alpha_{1,j}(\phi)$ and $\alpha_2=\alpha_{2,j}(\phi)$ are elements of the covariance matrix $A_j(\phi)$, whose analytic expression is given below in \eqref{alpha1} and \eqref{alpha2}. On the other hand, we have also the conditional densities
\begin{equation*}
D_{j,x,y,2} (t_1, t_2 | \nabla \tilde{\beta}_j(x)=\nabla \tilde{\beta}_j(y)=\mathbf{0})=%
\frac{1}{2 \pi} \frac{1}{\sqrt{\text{det} ( \Sigma_{2,j}(\phi))}} \exp\{-
\frac 1 2 (t_1, t_2) \Sigma_{2,j}^{-1}(\phi) (t_1, t_2)^T \},
\end{equation*}
where
\begin{equation*}
\Sigma_{2,j}(\phi)=R_j(\phi)-E_j(\phi) A^{-1}_j(\phi) E^T_j(\phi).
\end{equation*}
Moreover
\begin{align*}
&D_{j,x,y,3} (\zeta_{x}, \zeta_{y} |\tilde{\beta}_j(x)=t_1, \tilde{\beta}_j(y)=t_2, \nabla
\tilde{\beta}_j(x)=\nabla \tilde{\beta}_j(y)=\mathbf{0}) \\
& =\frac{1}{(2 \pi)^3} \frac{1}{\sqrt{ \text{det} (\Omega_j(\phi) )}} \exp{%
\{-\frac{1}{2} ((\zeta_{x}, \zeta_{y})-{\mu}_j(\phi,t_1,t_2))\Omega_j(%
\phi)^{-1} ((\zeta_{x}, \zeta_{y})-{\ \mu}_j(\phi,t_1,t_2))^T\}}
\end{align*}
where the mean vector and covariance matrix are given by
\begin{align*}
{\ \mu}_{j}(\phi,t_1,t_2)&=\left(
\begin{array}{cc}
D_j(\phi)^T & B_j(\phi)^T
\end{array}
\right) \left(
\begin{array}{cc}
R_j(\phi) & E_j(\phi) \\
E_j(\phi)^T & A_j(\phi)%
\end{array}
\right)^{-1} \left(
\begin{array}{c}
t_1 \\
t_2 \\
\mathbf{0} \\
\mathbf{0}%
\end{array}
\right)
\end{align*}
and
\begin{align*}
\Omega_j(\phi)&=C_j(\phi)- \left(
\begin{array}{cc}
D_j(\phi)^T & B_j(\phi)^T
\end{array}
\right) \left(
\begin{array}{cc}
R_j(\phi) & E_j(\phi) \\
E_j(\phi)^T & A_j(\phi)%
\end{array}
\right)^{-1} \left(
\begin{array}{c}
D_j(\phi) \\
B_j(\phi)%
\end{array}
\right),
\end{align*}
respectively. After the change of variables
\begin{equation*}
\frac{\sqrt{8}}{\sqrt{ \mathcal{B}_{4,j} }}(\zeta_{x}, \zeta_{y})=(\tilde{\zeta}_{x}, \tilde{\zeta}_{y}),
\end{equation*}
we have
\begin{align*}
&D_{j,x,y,3} (\zeta_{x}, \zeta_{y} |\tilde{\beta}_j(x)=t_1, \tilde{\beta}_j(y)=t_2, \nabla
\tilde{\beta}_j(x)=\nabla \tilde{\beta}_j(y)=\mathbf{0}) d \zeta_{x} d \zeta_{y} \\
& =\frac{1}{(2 \pi)^3} \frac{1}{\sqrt{ \text{det} (\Delta_j(\phi) )}} \exp{%
\{-\frac{1}{2} ((\tilde{\zeta}_{x}, \tilde{\zeta}_{y})-\tilde{\mu}%
_j(\phi,t_1,t_2))\Delta_j(\phi)^{-1} ((\tilde{\zeta}_{x}, \tilde{\zeta}_{y})-%
\tilde{ \mu}_j(\phi,t_1,t_2))^T\}} d \tilde{\zeta}_{x} d \tilde{\zeta}_{y}
\end{align*}
where
\begin{equation*}
\Delta_j(\phi)=\frac{{8}}{{\ \mathcal{B}_{4,j} }}\Omega_j(\phi).
\end{equation*}
The idea to conclude the proof is to write all the covariance matrix involved as (small) perturbations of their limiting values. For definiteness and simplicity, we consider the case where $2 < \gamma < 6$ and $p=1$; all the other parameter ranges can be dealt in an entirely analogous way, provided that $\gamma < 4 p+2$, as usually required for Mexican needlets. In particular, with a hard computation (which can be assisted by a computer), it is possible to show that the covariance matrix is given by
\begin{align*}
\Delta_{j}(\phi)=\left(
\begin{array}{cc}
\Delta_{1,j}(\phi) & \Delta_{2,j}(\phi) \\
\Delta_{2,j}(\phi) & \Delta_{1,j}(\phi)%
\end{array}
\right),
\end{align*}
with
\begin{align*}
\Delta_{1,j}(\phi)&= \left(
\begin{array}{ccc}
3 -v(\gamma) & 0 & 1- v(\gamma) \\
0 & 1
& 0 \\
1 - v(\gamma)  & 0 & 3  - v(\gamma)
\end{array}
\right)+ \left(
\begin{array}{ccc}
 \tilde{a}_{1,j}(\phi) -2 \frac{{\cal B}_{2,j}^2}{{\cal B}_{4,j}} +  v(\gamma) & 0 & \tilde{a}_{4,j}(\phi) -2 \frac{{\cal B}_{2,j}^2}{{\cal B}_{4,j}} +  v(\gamma) \\
0 & \tilde{a}_{2,j}(\phi)
& 0 \\
 \tilde{a}_{4,j}(\phi) -2 \frac{{\cal B}_{2,j}^2}{{\cal B}_{4,j}} +  v(\gamma) & 0 &  \tilde{a}_{3,j}(\phi) -2 \frac{{\cal B}_{2,j}^2}{{\cal B}_{4,j}} +  v(\gamma)
\end{array}
\right),\\
\Delta_{2,j}(\phi) &= \left(
\begin{array}{ccc}
\tilde{a}_{5,j}(\phi)  & 0 & \tilde{a}_{8,j}(\phi) \\
0 & \tilde{a}_{6,j}(\phi) & 0 \\
\tilde{a}_{8,j}(\phi) & 0 & \tilde{a}_{7,j}(\phi)
\end{array}
\right),
\end{align*}
where we note that in view of \eqref{limj}, for $\gamma \in (2,6)$, we have
$$2 \frac{\mathcal{B}_{2,j}^2}{\mathcal{B}_{4,j}}
= v(\gamma)+O(B^{-2j}), \hspace{1cm} v(\gamma)=2  \frac{6-  \gamma}{8- \gamma} \in (0,  4 /3).$$
Likewise
\begin{equation*}
\Sigma_{2,j}(\phi)=  \left(
\begin{array}{cc}
1 & 0\\
0 & 1
\end{array}
\right)+\left(
\begin{array}{cc}
 a_{1,j}(\phi) & a_{2,j}(\phi) \\
a_{2,j}(\phi) & a_{1,j}(\phi)%
\end{array}
\right).
\end{equation*}
In the previous formulae, we have introduced the vectors $\mathbf{a}$ and $\tilde{\mathbf{a}}$ that collect
the perturbing elements of the covariance matrices $\Sigma_{2,j}(\phi)$ and $%
\Delta_{j}(\phi)$ respectively, i.e.,
\begin{equation*}
\mathbf{a}=\mathbf{a}_j(\phi)=({a}_{1,j}(\phi), {a}_{2,j}(\phi)), \hspace{1cm%
} \tilde{\mathbf{a}}=\tilde{\mathbf{a}}_j(\phi)=(\tilde{{a}}_{1,j}(\phi),
\dots \tilde{{a}}_{8,j}(\phi)).
\end{equation*}
In view of Lemma \ref{asymp_matrix}, these elements are such that there exists a constant $K_M>0$
$$ a_{k,j}(\phi),\tilde{a}_{i,j}(\phi) \le \frac{K_M}{(1+ j^{-1} B^j \phi)^M}, \hspace{2cm}k=1,2,\;  i=1, \dots,8.$$
In what follows, with a slight abuse of notation, we write the conditional
covariance matrices $\Sigma_{j,2}(\phi)$ and $\Delta_{j}(\phi)$ as a
function of $\mathbf{a}$ and $\tilde{\mathbf{a}}$, and the 2-point
correlation function $K_{2,j}$ as a function of the perturbing elements ${a}%
_{k,j}(\phi)$, $k=1,2$ and $\tilde{{a}}_{i,j}(\phi)$, $i=1, \dots,8$.
It is a classical result of perturbation theory (see i.e., \cite{kato}) that Gaussian expectations are
analytic functions of the perturbing elements ${a}_{k,j}(\phi)$ and $\tilde{{%
a}}_{i,j}(\phi)$, so we can expand them into a Taylor polynomial around $%
\mathbf{0}$. In particular, let
\begin{equation} \label{qa}
q(\mathbf{a}; t_1, t_2)=\frac{1}{\sqrt{\text{det} ( \Sigma_{2,j}(\mathbf{a}))}}
\exp\{- \frac 1 2 (t_1, t_2) \Sigma_{2,j}^{-1}(\mathbf{a}) (t_1, t_2)^T \}
\end{equation}
and
\begin{equation*}
\tilde{q}(\tilde{\mathbf{a}}; t_1, t_2)=\frac{1}{\sqrt{ \text{det} (\Delta_j(%
\tilde{\mathbf{a}}) ) )}} \exp{\{-\frac{1}{2} ((\tilde{\zeta}_{x}, \tilde{%
\zeta}_{y})-\tilde{\mu}_j(\phi,t_1,t_2))\Delta_j(\tilde{\mathbf{a}})^{-1} ((%
\tilde{\zeta}_{x}, \tilde{\zeta}_{y})-\tilde{ \mu}_j(\phi,t_1,t_2))^T\}}.
\end{equation*}
As $||\mathbf{a}|| \to 0$ we have the expansions
\begin{align}
q(\mathbf{a}; t_1, t_2)&=q(\mathbf{0}; t_1, t_2)+ \sum_{i=1}^2 a_i \frac{%
\partial}{\partial a_i} q(\mathbf{0}; t_1, t_2) + a_1 a_2 \frac{\partial^2}{%
\partial a_1 a_2} q(\mathbf{0}; t_1, t_2) \nonumber \\
&\;\;+ \frac 1 2 \sum_{i=1}^2 a_i^2 \frac{\partial^2}{\partial a_i^2} q(%
\mathbf{0}; t_1, t_2) +O(r(t_1,t_2) ||\mathbf{a}||^3), \label{qat}
\end{align}
and
\begin{align*}
\tilde{q}(\tilde{\mathbf{a}}; t_1, t_2)&=\tilde{q}(\mathbf{0}; t_1, t_2)+
\sum_{i=1}^8 \tilde{a}_i \frac{\partial}{\partial \tilde{a}_i} \tilde{q}(%
\mathbf{0}; t_1, t_2) + \sum_{i \ne j} \tilde{a}_i \tilde{a}_j \frac{%
\partial^2}{\partial \tilde{a}_i \tilde{a}_j} \tilde{q}(\mathbf{0}; t_1, t_2)
\\
&\;\;+ \frac 1 2 \sum_{i=1}^8 \tilde{a}_i^2 \frac{\partial^2}{\partial
\tilde{a}_i^2} \tilde{q}(\mathbf{0}; t_1, t_2) +O(r(t_1,t_2) ||\mathbf{%
\tilde{a}}||^3).
\end{align*}
In view of Lemma \ref{asymp_matrix} and the analytic expressions of
the perturbing elements $a_i$ and $\tilde{a}_i$ we immediately have the following bounds
\begin{align} \label{5:57}
\int_{C/B^j}^{\pi/2} a_{i,j} (\phi)  \sin \phi \, d \phi,\; \int_{C/B^j}^{\pi/2} \tilde{a}_{k,j}(\phi) \sin \phi \, d \phi =O( j^2 B^{-2j}).
\end{align}
Now formula \eqref{finalr} follows by observing that
\begin{align*}
&\frac{16 \pi^2}{2^4 8^2 \pi^6} \frac{\mathcal{B}^2_{4,j} }{\mathcal{B}%
^2_{2,j}} \int_{C/B^j}^{\pi/2} \iint_{\mathbb{R}^3 \times \mathbb{R}^3} |%
\tilde{\zeta}_{x,1} \tilde{\zeta}_{x,3} -\tilde{\zeta}_{x,2}^2| |\tilde{\zeta%
}_{y,1} \tilde{\zeta}_{y,3} -\tilde{\zeta}_{y,2}^2| \frac{1}{\sqrt{ ( 1- 4
\frac{\alpha_1^2}{\mathcal{B}_{2,j}^2}) (1- 4 \frac{\alpha_2^2}{\mathcal{B}%
_{2,j}^2}) }}  q(\mathbf{0}; t_1, t_2) \\
& \;\; \times \tilde{q}({\mathbf{0}}; t_1, t_2) \sin \phi d \phi \tilde{\zeta}_{x,1} d \tilde{\zeta}_{x,2} d \tilde{\zeta}_{x,3} d \tilde{\zeta}_{y,1} \tilde{\zeta}_{y,2} d \tilde{\zeta}_{y,3}-
\frac{1}{4} \frac{\mathcal{B}^2_{4,j} }{\mathcal{B}^2_{2,j}} p(t_1)
p(t_2)\\
&\le c_1(t_1,t_2) \frac{\mathcal{B}^2_{4,j} }{\mathcal{B}^2_{2,j}} j^2 B^{-2j},
\end{align*}
because the leading terms cancel with the centring factor, and all the other components are bounded in view of \eqref{qa}, \eqref{qat} and \eqref{5:57}. This concludes the analysis of the long-range components.

\subsection{Proof of the short-range asymptotic bound \eqref{BB}}

\label{short}

We bound now the contribution of the terms that are at a smaller distance
than $C/B^j$. As in \cite{cmw2014}, we can use again the Kac-Rice metatheorem to show that
there exists a constant $c>0$ such that for every nice domain
$\mathcal{D} \subseteq \mathbb{S}^2$ contained in some spherical cap of
radius $c/B^j$, one has
\begin{equation}
\mathbb{E}[ \mathcal{N}^c(\tilde{\beta}_j; \mathcal{D},I) (\mathcal{N}^c(\tilde{\beta}_j;
\mathcal{D},I)-1) ]=\iint_{\mathcal{D} \times \mathcal{D}} \iint_{I \times
I} K_{2,j} (x,y;t_1,t_2) dt_1 dt_2 d x d y.
\end{equation}
\label{sette.tre}

\noindent Let
\begin{equation}
\varepsilon =c/{B^j};  \label{numberedformula}
\end{equation}
then we have

\begin{align}
\text{Var}\left( \mathcal{N}^{c}(\tilde{\beta}_j ;{\cal V}(\xi _{\varepsilon ,i}),I)\right)
&=\iint_{{\cal V}(\xi _{\varepsilon ,i})\times {\cal V}(\xi _{\varepsilon
,i})}\iint_{I\times I}K_{2,j}(x,y;t_{1},t_{2})dt_{1}dt_{2}dxdy  \notag \\
&\;\;+\mathbb{E}\left[ \mathcal{N}^{c}(\tilde{\beta}_j ;{\cal V}(\xi _{\varepsilon ,i}),I)%
\right] -\left(\mathbb{E}\left[ \mathcal{N}^{c}(\tilde{\beta}_j ;{\cal V}(\xi
_{\varepsilon,i}),I)\right]\right)^{2}.  \label{lario1}
\end{align}
It is easy to see that
\begin{align}
\mathbb{E}\left[ \mathcal{N}^{c}(\tilde{\beta}_j;V(\xi _{\varepsilon ,i}),I)\right]
\leq \mathbb{E}\left[ \mathcal{N}^{c}(\tilde{\beta}_j ;B(\xi _{\varepsilon
,i};\varepsilon ),I)\right] \leq \pi \varepsilon ^{2} B^{2j}=O(1),
\label{lario3}
\end{align}
by \eqref{numberedformula}. Moreover we have the following result

\begin{lemma}
\label{b} There exists a constant $c>0$ such that, for $d(x,y)<c/B^j$, one
has
\begin{align*}
\iint_{I \times I} K_{2,j}(x,y;t_{1},t_{2})dt_{1}dt_{2} \leq O(B^{4 j})
\end{align*}
with the constant involved in the $O$-notation universal.
\end{lemma}

\begin{proof}\;
Since $K_{2,j}(x,y;t_1,t_2)$ is nonnegative for all the values of its
arguments, it is enough to study the rate of the following uniform bound in $%
I$:
\begin{align*}
K_{2,j} (\phi,\mathbb{R},\mathbb{R})= \iint_{\mathbb{R} \times \mathbb{R}}
K_{2,j} (\phi; t_1,t_2) d t_1 d t_2;
\end{align*}
for $\phi<c/B^j$. With the change of variable $\phi=\psi/B^j$, we need to
study
\begin{align*}
K_{2,j}(\psi; \mathbb{R},\mathbb{R})&=\frac{1}{(2 \pi)^2} \frac{1}{\sqrt{%
\text{det} (A_{j}(\psi)) }} \frac{\mathcal{B}_{4,j}^2}{8^2} \iint_{\mathbb{R}%
^3 \times \mathbb{R}^3} |\tilde{\zeta}_{x,1} \tilde{\zeta}_{x,3} -\tilde{%
\zeta}_{x,2}^2| |\tilde{\zeta}_{y,1} \tilde{\zeta}_{y,3} -\tilde{\zeta}%
_{y,2}^2| \\
& \times \frac{1}{(2 \pi)^3} \frac{1}{\sqrt{ \text{det} (H_j(\psi)) }}
\exp{\ \{ -\frac{1}{2} \mathbf{\tilde{\zeta}} H_{j}(\psi)^{-1} \mathbf{%
\tilde{\zeta}}^t \}} d \tilde{\zeta}_{x,1} d \tilde{\zeta}_{x,2} d \tilde{%
\zeta}_{x,3} d \tilde{\zeta}_{y,1} d \tilde{\zeta}_{y,2} d \tilde{\zeta}%
_{y,3}.
\end{align*}
First note that
\begin{align*}
\frac{1}{\sqrt{ \text{det} (A_j(\psi)) }} &=\frac{1}{(2 \pi)^{2}} \frac{1}{
\frac{1}{4} \frac{1}{2^{4j}} \mathcal{B}_{2,j}^2 \sqrt{( 1-2^2\; 2^{4j}
\alpha^2_1(\psi)/\mathcal{B}_{2,j}^2) ( 1-2^2 \; 2^{4j} \alpha^2_2(\psi) /%
\mathcal{B}_{2,j}^2) }} \\
& \le \frac{1}{(2 \pi)^{2}} \frac{1}{ \frac{1}{4} \frac{1}{2^{4j}} \mathcal{B%
}_{2,j}^2 \sqrt{ ( 1-2\; 2^{2j} \alpha_1(\psi)/\mathcal{B}_{2,j}) (
1-2\;2^{2j} \alpha_2(\psi) /\mathcal{B}_{2,j}) }}
\end{align*}
and by Taylor expanding $\alpha_1$ and $\alpha_2$ (which again can be assisted by a computer), we obtain
\begin{align*}
( 1-2\; 2^{2j} \alpha_1(\psi)/\mathcal{B}_{2}) ( 1-2\;2^{2j}
\alpha_2(\psi) /\mathcal{B}_{2}) =O(\psi^4).
\end{align*}
It is easy to check that
\begin{align*}
K_{2,j}(\psi; \mathbb{R}, \mathbb{R}) \le \mathbb{E}[|X_1 X_3| |Y_1
Y_3|+|X_1 X_3| Y_2^2+|Y_1 Y_3| X_2^2+Y_2^2 X_2^2]
\end{align*}
where $(X_1,X_2,X_3,Y_1,Y_2,Y_3)$ is a centred Gaussian with covariance
matrix $H_j(\psi)$
\begin{align*}
H_j(\psi)&=C_j(\psi)- B_j(\psi)^t A_j(\psi)^{-1} B_j(\psi) =\left(
\begin{array}{cc}
H_{1,j}(\psi) & H_{2,j}(\psi) \\
H_{2,j}(\psi) & H_{1,j}(\psi)%
\end{array}
\right)
\end{align*}
Because we shall use a Cauchy-Schwartz bound, it is enough to focus on the blocks on the main diagonal, which we write as
\begin{align*}
H_{1,j}(\psi) =\left(
\begin{array}{lll}
h_{1,j}(\psi) & 0 & h_{4,j}(\psi) \\
0 & h_{2,j}(\psi) & 0 \\
h_{4,j}(\psi) & 0 & h_{3,j}(\psi);
\end{array}
\right)
\end{align*}
where the analytic expressions of the $h_{i,j}(\psi)$ are derived by a computer assisted computation. Now, by Cauchy-Schwartz, we have
\begin{align*}
&\mathbb{E}[|X_1 X_3| |Y_1 Y_3|+|X_1 X_3| Y_2^2+|Y_1 Y_3| X_2^2+Y_2^2 X_2^2]
\\
& \le 3 h_{1,j}(\psi) h_{3,j}(\psi)+6 (h_{1,j}(\psi))^{1/2}
(h_{3,j}(\psi))^{1/2} h_{2,j}(\psi)+3 h_{2,j}(\psi)^2.
\end{align*}
and by Taylor expanding the $h_{i,j}(\psi)$ around $\psi=0$ (again by a computer assisted computation) we have
\begin{align*}
&\mathbb{E}[|X_1 X_3| |Y_1 Y_3|+|X_1 X_3| Y_2^2+|Y_1 Y_3| X_2^2+Y_2^2 X_2^2]
\\
&\le\frac{3 \; 2^{-18 j-7} \left(2 \mathcal{B}_{2,j}-\mathcal{B}_{4,j}\right)
\mathcal{B}_{4,j} \left(2^{8 j+1} \mathcal{B}_{2,j} \left(6 \mathcal{B}%
_{2,j}-4 \mathcal{B}_{4,j}+5 \mathcal{B}_{6,j}\right)-9 \; 2^{8 j} \mathcal{B%
}_{4,j}^2\right) \psi ^2}{\mathcal{B}_{2,j} \left(2 \mathcal{B}_{2,j}-3
\mathcal{B}_{4,j}\right)}+O\left(\psi ^3\right)
\end{align*}
so that the numerator is uniformly bounded by terms of order $\psi^2$ and
\begin{align*}
K_{2,j}(\psi; \mathbb{R}, \mathbb{R})=\frac{O(\psi^2)}{O(\psi^2)} \frac{%
\mathcal{B}_{4,j}^2}{\mathcal{B}_{2,j}^2}=O(B^{4j}).
\end{align*}
\end{proof}

\noindent By Lemma \ref{b}, we have
\begin{equation}
\int_{{\cal V}(\xi _{\varepsilon ,i})\times {\cal V}(\xi _{\varepsilon ,i})}\int_{I\times
I}K_{2,j}(x,y;t_{1},t_{2})dt_{1}dt_{2}dxdy\leq B^{4 j}\cdot (\pi \varepsilon
^{2})^{2}=O(1),  \label{nonnumberedyet}
\end{equation}
again by \eqref{numberedformula}. Substituting the estimates \eqref{lario3}
and \eqref{nonnumberedyet} into \eqref{lario1} yields
\begin{equation}
\text{Var}\left( \mathcal{N}^{c}(\tilde{\beta}_j ;{\cal V}(\xi _{\varepsilon ,i}),I)\right)
=O(1).  \label{seriousnumber}
\end{equation}
and, by \eqref{seriousnumber},
\begin{equation}
\left\vert \text{Cov}\left( \mathcal{N}^{c}(\tilde \beta_j ;{\cal V}(\xi _{i,\varepsilon
}),I),\mathcal{N}^{c}(\tilde{\beta}_j ;{\cal V}(\xi _{j,\varepsilon }),I)\right)
\right\vert \leq O(1).  \label{csbound}
\end{equation}
As there are $O(B^{2 j})$ pairs of Voronoi cells at distance smaller or
equal than $C/{B^j}$, \eqref{csbound} implies that the contribution of this
range to \eqref{mir1} is
\begin{equation*}
\sum_{d({\cal V}(\xi _{i,\varepsilon }),{\cal V}(\xi _{k,\varepsilon })) \le C/{B^j}
}\left\vert \text{Cov}\left( \mathcal{N}^{c}(\tilde{\beta}_j ;{\cal V}(\xi _{i,\varepsilon
}),I),\mathcal{N}^{c}(\tilde{\beta}_j ;{\cal V}(\xi _{k,\varepsilon }),I)\right)
\right\vert =O(B^{2 j}).
\end{equation*}

\section{Proof of auxiliary results: covariance matrices} \label{cov_mat_section}

In this section we evaluate the covariance matrix $%
\Sigma_j (x,y)$ of the 12-dimensional Gaussian vector
\begin{equation*}
(\tilde \beta_j (x),\tilde \beta_j(y), \nabla \tilde \beta_j(x) , \nabla \tilde \beta_j(y), \nabla^2
\tilde \beta_j(x), \nabla^2 \tilde \beta_j(y)),
\end{equation*}
which combines level, gradient and elements of the Hessian evaluated at $x$
and $y$. The computations do not require sophisticated arguments, other than
iterative derivations of Legendre polynomials. Note that $\Sigma_j (x,y)$
depends only on the geodesic distance $\phi = d(x, y)$, so, abusing
notation, we shall write $\Sigma_j (x,y)=\Sigma_j(\phi)$. It is convenient
to write $\Sigma_j (x,y)$ in block-diagonal form, i.e.
\begin{equation*}
\Sigma_j(\phi)=\left(
\begin{array}{ccc}
R_j(\phi) & E_j(\phi) & D_j(\phi) \\
E^{T}_j(\phi) & A_j(\phi) & B_j(\phi) \\
D^{T}_j(\phi) & B^{T}_j(\phi) & C_j(\phi)%
\end{array}
\right).
\end{equation*}
In what follows we use the notation
$$b_{\ell,j}=\frac{1}{ \sum_{\ell=B^{j-1}}^{B^{j+1}} b_p^2(\frac{\ell}{B^j})
C_\ell \frac{2\ell+1}{4 \pi} } b_p^2(\frac{\ell}{B^j})
C_\ell \frac{2\ell+1}{4 \pi}.$$

\begin{itemize}
\item $R_j(\phi)$ is the covariance matrix of $(\tilde \beta_j(x), \tilde \beta_j(y))$:
\end{itemize}
\begin{equation*}
R_j(\phi)_{2\times 2}=\left(
\begin{array}{cc}
1 & \rho_j(\phi)  \\
\rho_j(\phi) & 1
\end{array}
\right), \hspace{1cm} \rho_j(\phi)=\sum_{\ell=B^{j-1}}^{B^{j+1}} b_{\ell,j} P_\ell(\cos \phi).
\end{equation*}

\begin{itemize}
\item $E_j(\phi)$ is the $2 \times 4$ matrix
\end{itemize}
\begin{align*}
E_j (x,y)_{2\times 4}& 
=\left(
\begin{array}{cccc}
0 & 0 & 0 & \epsilon _j (\phi ) \\
0 & -\epsilon _j (\phi ) & 0 & 0%
\end{array}
\right), \hspace{1cm}\epsilon _j(\phi )= \sin \phi \sum_{\ell=B^{j-1}}^{B^{j+1}}
b_{\ell,j}P_{\ell }^{\prime }(\cos \phi ).
\end{align*}
\begin{itemize}
\item $D_j(\phi)$ is given by
\end{itemize}
\begin{align*}
D_j(x,y)_{2\times 6} &
=\left(
\begin{array}{cccccc}
-\frac{1}{2} \mathcal{B}_{2,j} & 0 & -\frac{1}{2} \mathcal{B}_{2,j} &
-\delta _{\ell }(\phi ) & 0 & -\delta _{\ell }(\phi ) \\
-\delta _{\ell }(\phi ) & 0 & -\delta _{\ell }(\phi ) & -\frac{1}{2}
\mathcal{B}_{2,j} & 0 & -\frac{1}{2} \mathcal{B}_{2,j}%
\end{array}%
\right),
\end{align*}%
with elements
\begin{equation*}
\delta _{\ell }(\phi )=\cos \phi \sum_{\ell } b_{\ell,j} P_{\ell }^{\prime
}(\cos \phi ).
\end{equation*}
\begin{itemize}
\item A is the $4 \times 4$ matrix given by
\end{itemize}
\begin{align*}
A_j(x,y)_{4\times 4}
=\left(
\begin{array}{cccccc}
\frac{1}{2} \mathcal{B}_{2,j} & 0 & \alpha _{1,j }(\phi ) & 0 &  &  \\
0 & \frac{1}{2} \mathcal{B}_{2,j} & 0 & \alpha _{2,j }(\phi ) &  &  \\
\alpha _{1,j }(\phi ) & 0 & \frac{1}{2} \mathcal{B}_{2,j} & 0 &  &  \\
0 & \alpha _{2,j }(\phi ) & 0 & \frac{1}{2} \mathcal{B}_{2,j} &  &
\end{array}%
\right) ,
\end{align*}%
where
\begin{align}
\alpha _{1,j }(\phi )&=\sum_{\ell } b_{\ell,j} P_{\ell }^{\prime }(\cos \phi
),\label{alpha1}\\
\alpha _{2,j}(\phi )&=-\sin ^{2}\phi \sum_{\ell } b_{\ell,j}
P_{\ell }^{\prime \prime }(\cos \phi )+\cos \phi \sum_{\ell } b_{\ell,j}
P_{\ell }^{\prime }(\cos \phi ). \label{alpha2}
\end{align}
\begin{itemize}
\item $B$ is given by
\end{itemize}
\begin{align*}
B_j(x,y)_{4\times 6}
=\left(
\begin{array}{cc}
\mathbf{0} & b_j(\phi) \\
-b_j(\phi) & \mathbf{0}%
\end{array}
\right)
\end{align*}
with
\begin{align*}
b_j(\phi) =\left(
\begin{array}{ccc}
0 & \beta _{1,j }(\phi ) & 0 \\
\beta _{2,j }(\phi ) & 0 & \beta _{3,j}(\phi )%
\end{array}
\right),
\end{align*}
\begin{eqnarray*}
\beta _{1,j }(\phi ) &=&\sin \phi \sum_{\ell } b_{\ell,j} P_{\ell }^{\prime
\prime }(\cos \phi ), \\
\beta _{2,j}(\phi ) &=&\sin \phi \cos \phi \sum_{\ell } b_{\ell,j} P_{\ell
}^{\prime \prime }(\cos \phi )+\sin \phi \sum_{\ell } b_{\ell,j} P_{\ell
}^{\prime }(\cos \phi ), \\
\beta _{3,j }(\phi ) &=&-\sin ^{3}\phi \sum_{\ell } b_{\ell,j} P_{\ell
}^{\prime \prime \prime }(\cos \phi )+3\sin \phi \cos \phi \sum_{\ell }
b_{\ell,j} P_{\ell }^{\prime \prime }(\cos \phi )+\sin \phi \sum_{\ell }
b_{\ell,j} P_{\ell }^{\prime }(\cos \phi ).
\end{eqnarray*}

\begin{itemize}
\item Finally for $C$ we have
\end{itemize}
\begin{align*}
C_j (x,y)_{6\times 6}
=\left(
\begin{array}{cc}
c_j(0) & c_j(\phi) \\
c_j(\phi) & c_j(0)%
\end{array}
\right)
\end{align*}
where
\begin{align*}
c_j(0) =\left(
\begin{array}{ccc}
\frac{3}{8} \mathcal{B}_{4,j} -\frac{1}{4} \mathcal{B}_{2,j} & 0 & \frac{1}{8%
} \mathcal{B}_{4,j} +\frac{1}{4} \mathcal{B}_{2,j} \\
0 & \frac{1}{8} \mathcal{B}_{4,j} -\frac{1}{4} \mathcal{B}_{2,j} & 0 \\
\frac{1}{8} \mathcal{B}_{4,j} +\frac{1}{4} \mathcal{B}_{2,j} & 0 & \frac{3}{8%
} \mathcal{B}_{4,j} -\frac{1}{4} \mathcal{B}_{2,j}%
\end{array}%
\right),
\end{align*}
and
\begin{align*}
c_j(\phi) =\left(
\begin{array}{ccc}
\gamma _{1,j}(\phi ) & 0 & \gamma _{3,j}(\phi ) \\
0 & \gamma _{2,j}(\phi ) & 0 \\
\gamma _{3,j}(\phi ) & 0 & \gamma _{4,j }(\phi )%
\end{array}%
\right),
\end{align*}%
with
\begin{align*}
\gamma _{1,j}(\phi )& =(2+\cos ^{2}\phi ) \sum_{\ell } b_{\ell,j} P_{\ell
}^{\prime \prime }(\cos \phi )+\cos \phi \sum_{\ell } b_{\ell,j} P_{\ell
}^{\prime }(\cos \phi ) \\
\gamma _{2,j }(\phi )& =-\sin ^{2}\phi \sum_{\ell } b_{\ell,j} P_{\ell
}^{\prime \prime \prime }(\cos \phi )+\cos \phi \sum_{\ell } b_{\ell,j}
P_{\ell }^{\prime \prime }(\cos \phi ) \\
\gamma _{3,j}(\phi )& =-\sin ^{2}\phi \cos \phi \sum_{\ell } b_{\ell,j}
P_{\ell }^{\prime \prime \prime }(\cos \phi )+(-2\sin ^{2}\phi +\cos
^{2}\phi ) \sum_{\ell } b_{\ell,j} P_{\ell }^{\prime \prime }(\cos \phi ) \\
&\;\;+\cos \phi \sum_{\ell } b_{\ell,j} P_{\ell }^{\prime }(\cos \phi ), \\
\gamma _{4,j}(\phi )& =\sin ^{4}\phi \sum_{\ell } b_{\ell,j} P_{\ell
}^{\prime \prime \prime \prime }(\cos \phi )-6\sin ^{2}\phi \cos \phi
\sum_{\ell } b_{\ell,j} P_{\ell }^{\prime \prime \prime }(\cos \phi ) \\
&\;\;+(-4\sin ^{2}\phi +3\cos ^{2}\phi ) \sum_{\ell } b_{\ell,j} P_{\ell
}^{\prime \prime }(\cos \phi )+\cos \phi \sum_{\ell } b_{\ell,j} P_{\ell
}^{\prime }(\cos \phi ).
\end{align*}

It is a well-known fact in the theory of stationary and isotropic stochastic
processes that the covariance of derivative fields can be evaluated by means
of derivatives of the covariance functions; this issue is discussed for
instance in \cite{RFG}, p. 268. Hence, for the results to follow, we
shall need to control the asymptotic behaviour of higher-order derivatives
of this covariance function. These results are collected in the following
Proposition.

\begin{proposition} \label{geller_t}
For all nonnegative integers $a,b,c,d\in \mathbb{N}_{+},$ there
exist $K>0$ such that, for all $x,y\in S^{2}$
\begin{equation*}
\frac{\partial ^{a}}{\partial \theta _{x}^{a}}\frac{\partial ^{b}}{\partial
\theta _{y}^{b}}\frac{\partial ^{c}}{\partial \phi _{x}^{c}}\frac{\partial
^{d}}{\partial \phi _{y}^{d}}\mathbb{E}[\tilde{\beta}_j(x)\tilde{\beta}%
_{j,p}(y)]\leq B^{j(a+b+c+d)}\frac{K}{(1+j^{-1}B^{j}d(x,y))^{4p+2-\gamma }}.
\end{equation*}
\end{proposition}

\begin{proof} \;
The result is a simple consequence of equation (8) in \cite{gm1},
from which we have that
\begin{align*}
\frac{\partial ^{a}}{\partial \theta _{x}^{a}}\frac{\partial ^{b}}{\partial
\theta _{y}^{b}}\frac{\partial ^{c}}{\partial \phi _{x}^{c}}\frac{\partial
^{d}}{\partial \phi _{y}^{d}}\mathbb{E}[\beta _j(x)\beta _j(y)]&
=\sum_{\ell =1}^{\infty }b^{2}\Big(\frac{\ell }{B^{j}};p\Big)C_{\ell }\frac{%
2\ell +1}{4\pi }\frac{\partial ^{a}}{\partial \theta _{x}^{a}}\frac{\partial
^{b}}{\partial \theta _{y}^{b}}\frac{\partial ^{c}}{\partial \phi _{x}^{c}}%
\frac{\partial ^{d}}{\partial \phi _{y}^{d}}P_{\ell }(\langle x,y\rangle ) \\
& \leq B^{j(a+b+c+d)}\frac{K_{M}}{(1+j^{-1}B^{j}d(x,y))^{4p+2-\gamma }}%
\sum_{\ell =1}^{\infty }b^{2}\Big(\frac{\ell }{B^{j}};p\Big)C_{\ell }\frac{%
2\ell +1}{4\pi },
\end{align*}%
$a,b,c,d\in \mathbb{N}$, so that
\begin{equation*}
\frac{\partial ^{a}}{\partial \theta _{x}^{a}}\frac{\partial ^{b}}{\partial
\theta _{y}^{b}}\frac{\partial ^{c}}{\partial \phi _{x}^{c}}\frac{\partial
^{d}}{\partial \phi _{y}^{d}}\mathbb{E}[\tilde{\beta}_j(x)\tilde{\beta}%
_{j,p}(y)]\leq B^{j(a+b+c+d)}\frac{K_{M}}{(1+j^{-1}B^{j}d(x,y))^{4p+2-\gamma
}}.
\end{equation*}
\end{proof}
From Proposition \ref{geller_t} it immediately follows that
\begin{lemma}
\label{asymp_matrix} There exists a constant $K_M > 0$ such that, for all $%
x,y \in \mathbb{S}^2$ we have
\begin{align*}
&\rho_j(\phi) \le \frac{K_M}{(1+j^{-1} B^j \phi)^M}, \\
& \epsilon _j(\phi ), \delta _j (\phi ), \alpha _{1,j }(\phi ), \beta _{1,j
}(\phi ) \le B^j \frac{K_M}{(1+ j^{-1} B^j \phi)^M}, \hspace{1cm} \alpha_{2,j}(\phi
), \beta _{2,j}(\phi ), \gamma _{1,j}(\phi ) \le B^{2 j} \frac{K_M}{%
(1+ j^{-1} B^j \phi)^M}, \\
&\beta _{3,j }(\phi ), \gamma _{2,j }(\phi ), \gamma _{3,j}(\phi ) \le B^{3j} \frac{%
K_M}{(1+ j^{-1} B^j \phi)^M} , \hspace{1.7cm}
\gamma _{4,j}(\phi ) \le B^{4 j} \frac{K_M}{(1+ j^{-1} B^j \phi)^M}.
\end{align*}
\end{lemma}

\noindent Dan Cheng and Armin Schwartzman\\
Division of Biostatistics, University of California, San Diego\\
dcheng2@ncsu.edu; armins@ucsd.edu

\smallskip

\noindent Valentina Cammarota, Yabebal Fantaye and Domenico Marinucci\\
Department of Mathematics, University of Rome Tor Vergata\\
cammarot@mat.uniroma2.it; fantaye@mat.uniroma2.it; marinucc@mat.uniroma2.it

\end{document}